\documentclass[10pt]{amsart}

\usepackage{amsmath,amssymb,amsthm,mathtools,mathrsfs}
\usepackage{enumitem}
\usepackage{microtype}
\usepackage[colorlinks=true,linkcolor=blue,citecolor=blue,urlcolor=blue]{hyperref}
\usepackage[backend=biber,style=numeric,sorting=nyt,maxbibnames=10]{biblatex}
\allowdisplaybreaks
\setlist{itemsep=1pt,topsep=3pt,parsep=1pt}

\newtheorem{theorem}{Theorem}[section]
\newtheorem{proposition}[theorem]{Proposition}
\newtheorem{lemma}[theorem]{Lemma}
\newtheorem{corollary}[theorem]{Corollary}
\theoremstyle{definition}
\newtheorem{definition}[theorem]{Definition}
\theoremstyle{remark}
\newtheorem{remark}[theorem]{Remark}

\newcommand{\C}{\mathbb C}
\newcommand{\Cl}{\operatorname{Cl}}
\newcommand{\N}{\mathbb N}

\newcommand{\Ker}{\operatorname{Ker}}

\newcommand{\scrP}{\mathscr P}
\newcommand{\scrH}{\mathscr H}
\newcommand{\scrF}{\mathscr F}
\newcommand{\scrI}{\mathscr I}
\newcommand{\scrR}{\mathscr R}

\title[Stable $q$-Hermitian coordinate calculus]{Stable $q$-Hermitian coordinate calculus:\\Capelli--PBW transport, mixed polarization algebra, and localization obstructions}

\author{Baruch Schneider}
\address{Department of Mathematics, University of Ostrava, 70200 Ostrava, Czechia}
\email{baruch.schneider@osu.cz}

\author{Diana Schneiderov\'a}
\address{Department of Mathematics, University of Ostrava, 70200 Ostrava, Czechia}
\email{diana.schneiderova@osu.cz}

\author{Yifan Zhang}
\address{Department of Mathematics, University of Ostrava, 70200 Ostrava, Czechia; Department of Algebra, Charles University, 18675 Prague, Czechia; Department of Applied Mathematics, VSB--Technical University of Ostrava, 70800 Ostrava, Czechia}
\email{yifan.zhang@osu.cz}

\date{}

\begin{document}

\begin{abstract}
We extend a finite-support $q$-Hermitian radial calculus to
Clifford-valued coordinate polynomials.  For $N$ bosonic Hermitian vector
labels and finite fermionic support, the stable range $m\ge 2N$ admits a
canonical Gram--harmonic normal form by Howe separation.  A label--Capelli
inverse normalizes the Clifford contractions, while a finite
Wick--Chevalley conjugation lifts the divided-power scalar gauge to the full
Clifford PBW module.  Coupling the normalized contractions to contractible
fermionic seed complexes gives two conjugate polynomial-preserving
$q$-Hermitian coordinate families.  They are square-zero, anticommute within
each polarization, restrict exactly to the radial PBW calculus, have scalar
boundary trace $[2m-2n]_q$, and recover the classical Hermitian super Dirac
pair as $q\to1$.

For the mixed polarizations we introduce a relative triangular transport.
The Grassmann relations for the normalized Capelli contractions hold on the
full coordinate module and yield a labelwise factorization.  Hence all
higher-filtration mixed terms are finite subset products of explicit local
defects, and the mixed polarization algebra closes for arbitrary finite
label sets and finite fermionic support.  The classical complex structure is
common to both polarizations, whereas the two full transported
Clifford--Weyl products are distinct for $0<q<1$.

We also show that the fermionic Cartan denominator is forced in the natural
polynomial first-order Berezin--Weyl class and determine the minimal
factorwise Ore localization needed for the conjugate all-charge homotopies.
The $q$-dependent Cartan orbit is nonresonant for
$m\ge\max\{2N,n\}$.  Finally, fixed finite-order differential and finite
nonzero-shift realizations on the undeformed coordinate algebra are ruled
out.
\end{abstract}


\maketitle

\section{Introduction}\label{sec:intro}

Throughout $0<q<1$ and
\[
 [r]_q:=\frac{1-q^r}{1-q},\qquad
 \gamma_q:=\frac{1+q}{2},\qquad
 c_r:=\gamma_q[r]_{q^2}=\frac12[2r]_q.
\]
For a commuting variable $x$ and ratio $t\ne1$ we write
\[
 \delta_{x,t}F:=\frac{F(x)-F(tx)}{(1-t)x}.
\]

The radial algebra of abstract vector variables separates the invariant
algebraic content of Clifford analysis from a particular coordinate
realization and dimension \cite{Sommen1997}.  After adjoining a complex
structure, the Hermitian radial algebra encodes the Witt variables, the two
Hermitian vector derivatives and their Euler--Cartan relations
\cite{DeSchepperGuzmanSommen2017}.  In the classical superspace
representation these objects become local differential operators in commuting
and anticommuting complex coordinates
\cite{DeSchepperGuzmanSommen2018Hermitian,GuzmanAdan2018}.

There are several inequivalent ways to introduce $q$-calculus into Clifford
analysis.  An axiomatic radial deformation of the Dirac operator leads to a
$q$-Laplace operator, quantum-algebra structure and $q$-Clifford--Hermite
polynomials \cite{CoulembierSommen2010}.  More recently, in the
commuting-coordinate setting, a coordinatewise Jackson replacement was used
to construct an explicit $q$-Dirac operator together with $q$-Euler and
$q$-Gamma operators, Fischer decomposition and a $q$-Cauchy--Kovalevskaya
extension theorem \cite{ZimmermannBernsteinSchneider2025}.  The radial problem
studied here is different: a coordinatewise Jackson replacement does not
preserve the intrinsic radial subalgebra \cite{BarseghyanBorySchneiderZhang2026},
and at integral superdimension the natural projection of independent
orthogonal right $q$-vector derivatives does not descend to the Hermitian
quotient \cite{BoryReyesSchneiderSchneiderovaZhang2026}.  Thus the classical
2018 projection construction cannot simply be $q$-deformed inside the undeformed
coordinate algebra.
The difficulty is therefore not merely to write down a Jackson-type Dirac
operator, but to reconcile Hermitian polarization, exact radial restriction,
and closure of the mixed $+/-$ algebra on general coordinates.  The stable
construction below addresses this compatibility problem.

We therefore take the following flat radial calculus as the reference point.  The radial Gram powers are
transported by divided-power weights, while the finite fermionic boundary is
corrected by an explicit contractible seed.  The resulting scalar restrictions
are
\begin{align}
 \mathfrak D_{k,q}^{\rm sc}P
 &=\gamma_q Z_k^\dagger\,\delta_{H_{kk},q^2}P
   +\sum_{j\ne k}Z_j^\dagger\,\delta_{H_{kj},q}P,
 \label{eq:radial-target-minus}\\
 (\mathfrak D_{k,q}^{\dagger})^{\rm sc}P
 &=\gamma_q Z_k\,\delta_{H_{kk},q^2}P
   +\sum_{i\ne k}Z_i\,\delta_{H_{ik},q}P.
 \label{eq:radial-target-plus}
\end{align}
The purpose of the present paper is to extend this flat radial calculus to arbitrary Clifford-valued bosonic coordinate polynomials,
while retaining finite fermionic support and the exact $q$-superdimension
trace.  In particular, the flat radial operators used in the coordinate
intertwining theorem are defined below from the same divided-power and
fermionic data; no unpublished radial construction is used as a logical
input.

There are two different notions of ``coordinate extension''.  A classical
2018-type realization would be a fixed local differential calculus on an
ordinary superspace.  We do not obtain such a realization.  Instead we solve
the stable coordinate problem in the \emph{localized transported setting}: the bosonic coordinate part
is a polynomial module, the fermionic part is the finite tensor of
charge modules after an explicit factorwise Cartan Ore localization, and the
$q$-operators are obtained by polynomial-preserving spectral and triangular
transports.  Sections~\ref{sec:obstructions}--\ref{sec:conclusion} explain why
this distinction is part of the algebraic construction.

Although the final operators are written as conjugates of ungauged flat seeds,
the existence of the required conjugating maps on the full coordinate algebra
is not formal.  The scalar divided-power gauge does not preserve Clifford PBW
words, the two Hermitian polarizations require different triangular frames,
and the fermionic boundary cannot be implemented polynomially in the natural
first-order Berezin--Weyl class.  The Capelli--PBW factorization, finite Wick lift, relative mixed-algebra factorization and localization-minimality results below establish the transported construction on the full coordinate module.

The construction has three ingredients.  First, for $m\ge2N$, Howe harmonic
duality gives the unique scalar normal form
\[
 P=\sum_\alpha H^\alpha h_\alpha,
 \qquad
 \Delta_{ij}h_\alpha=0,
\]
so the radial divided-power weights extend canonically from the Gram algebra
to all scalar coordinate polynomials.  Second, a label--Capelli inverse turns
the coordinate Clifford contractions into the abstract Grassmann
contractions on every standard radial exterior word.  A finite Wick map then
corrects the scalar gauge on the full Clifford PBW module.  Third, these
normalized contractions couple to the finite contractible fermionic seed
complexes by one finite triangular conjugation.

The mixed polarization problem is not governed by one common triangular
frame.  Nevertheless, it reduces to a finite relative transport
\[
 \mathbb R_q=\mathbb U_q^+(\mathbb U_q^-)^{-1}.
\]
A central result is that the normalized Capelli contractions obey
the full Grassmann relations on the entire coordinate module.  It follows
that the triangular factors commute between different bosonic labels and
$\mathbb R_q$ factorizes labelwise.  Consequently the complete mixed
anticommutator is a finite product of local conjugation defects.  No additional
higher-label curvature tensor or new degree-dependent coefficients are required.

The fermionic localization is forced already in the one-pair problem.  The
required trace--Cartan boundary value yields the rational factor
$(B_{\rm f}-i)^{-1}$ in the natural polynomial first-order Berezin--Weyl
class.  Requiring both conjugate contracting homotopies on every charge forces
three shift-orbit families and no mixed denominators.  In the natural range
$m\ge n$ all $q$-dependent orbit factors are nonresonant on the
standard polynomial oscillator; only the $q$-independent Weyl orbit can
resonate on isolated shifted shells.

The presentation below is self-contained for the coordinate construction once
the standard classical Hermitian radial algebra is fixed.  The earlier
finite-support radial theory motivates the target formulas, but the flat
radial $q$-operators needed here are reconstructed in
Section~\ref{sec:coordinate} from the displayed divided-power gauge and
fermionic seed rather than taken as an external input:
Section~\ref{sec:fermionic} writes the charge blocks, inverse homotopies,
minimal localization and finite-support occupancy defect explicitly.  Likewise,
the mixed section gives both the factorized formula and, in
Appendix~\ref{app:mixed-normal-order}, the rank-one right-normal calculation
from which the local defects can be read without an auxiliary unpublished
calculation.

The paper is organized as follows.  Section~\ref{sec:scalar} fixes the
coordinate algebra and radial PBW sector and proves the stable scalar normal form.
Section~\ref{sec:fermionic} constructs the localized fermionic seed and its
finite-support tensorization.  Section~\ref{sec:capelli} gives the
Capelli--PBW and Wick lifts, and Section~\ref{sec:coordinate} assembles the
stable coordinate operators.  Section~\ref{sec:mixed} proves the complete
mixed algebra by relative transport and global Capelli--Grassmann
factorization.  Section~\ref{sec:obstructions} records the two locality
obstructions and the structural scope of the result.

We summarize the principal conclusions.  Precise versions are proved below.

\begin{theorem}[Stable coordinate construction, informal form]
\label{thm:intro-A}
Let $N,n<\infty$ and $m\ge2N$.  On arbitrary Clifford-valued bosonic
coordinate polynomials tensored with the factorwise localized finite-support
fermionic charge modules there exist conjugate families
$\mathbb D_{k,q}^\pm$ such that
\[
 (\mathbb D_{k,q}^\pm)^2=0,
 \qquad
 \{\mathbb D_{i,q}^\pm,\mathbb D_{j,q}^\pm\}=0\quad(i\ne j).
\]
They restrict exactly to the finite-support radial PBW calculus defined in
Section~\ref{sec:coordinate}, have scalar boundary trace $[2m-2n]_q$, are $U(m)$- and label-covariant, and converge to the
classical Hermitian super Dirac pair at $q=1$.
\end{theorem}

\begin{theorem}[Complete mixed algebra, informal form]
\label{thm:intro-B}
In the same range, the mixed matrix
$\{\mathbb D_{i,q}^+,\mathbb D_{j,q}^-\}$ is completely determined by the
ungauged mixed seed matrix and finitely many labelwise relative defects.  The
relative transport factorizes into commuting local factors, so every
higher-filtration term is a finite subset product of the same local data.
\end{theorem}

\begin{theorem}[Localization and product structure, informal form]
\label{thm:intro-C}
The one-pair boundary denominator is forced in the full charge-preserving
polynomial first-order seed class.  The factorwise shift-stable Cartan
localization generated by the three Weyl/trace orbit families is sufficient
and minimal for both conjugate all-charge homotopies.  The classical complex
structure is common to both polarizations, but the two full transported
Clifford--Weyl products are distinct for $0<q<1$.
\end{theorem}

The stable assumption $m\ge2N$ is essential to the present method.  Below it,
several-variable Fischer separation acquires determinantal and Dirac-complex
syzygies \cite{SabadiniStruppaSommenVanLancker2002,LavickaSoucek2017}; that
exceptional problem is not absorbed into a nonunique gauge and is left for a
separate treatment.

\section{Coordinate setup, radial calculus, and stable scalar transport}\label{sec:scalar}

\subsection{Classical Hermitian coordinates and conventions}
Fix $m,N\ge1$.  For $1\le k\le N$ let
$z_k=(z_{k1},\ldots,z_{km})$ and
$\bar z_k=(\bar z_{k1},\ldots,\bar z_{km})$ be independent complex
coordinate tuples.  Let $f_a,f_a^\dagger$ be a Witt basis,
\begin{equation}\label{eq:witt}
 \{f_a,f_b\}=\{f_a^\dagger,f_b^\dagger\}=0,
 \qquad
 \{f_a,f_b^\dagger\}=\delta_{ab}.
\end{equation}
Put
\begin{equation}\label{eq:coordinate-vectors}
 Z_k:=\sum_{a=1}^m z_{ka}f_a,
 \qquad
 Z_k^\dagger:=\sum_{a=1}^m\bar z_{ka}f_a^\dagger,
 \qquad
 H_{ij}:=\sum_{a=1}^m z_{ia}\bar z_{ja},
\end{equation}
and
\begin{equation}\label{eq:classical-dirac}
 \partial_{Z_k}:=\sum_{a=1}^m f_a^\dagger\partial_{z_{ka}},
 \qquad
 \partial_{Z_k^\dagger}:=\sum_{a=1}^m f_a\partial_{\bar z_{ka}}.
\end{equation}
The same-polarization squares and anticommutators vanish.

Throughout, $\ddagger$ denotes the \emph{order-preserving} antilinear
Hermitian involution on the represented coordinate--operator algebra.  It
exchanges holomorphic and antiholomorphic coordinate/operator data and later
reverses the fermionic charge.  Thus
\begin{equation}\label{eq:ddagger-order}
 (FG)^\ddagger=F^\ddagger G^\ddagger.
\end{equation}
This is not the usual order-reversing Clifford anti-involution.  The dagger on
$f_a^\dagger$ denotes the conjugate Witt generator; it does not mean that
products are reversed by $\ddagger$.  We make this convention explicit
because the two transported products in Section~\ref{sec:mixed} are exchanged
by this order-preserving conjugation.

Let
\[
 \scrP^{\rm sc}_{m,N}:=
 \C[z_{ia},\bar z_{ja}:1\le i,j\le N,\ 1\le a\le m]
\]
and let
\[
 \scrP^{\rm Cl}_{m,N}:=\scrP^{\rm sc}_{m,N}\otimes\Cl_{2m}^{\C}
\]
denote the corresponding Clifford-valued polynomial module.  We only use the
Clifford algebra through the Witt relations \eqref{eq:witt}; no choice of a
matrix representation is required.

\subsection{The represented radial PBW sector}
The coordinate theorem is required to agree with the radial calculus defined
below on its natural PBW sector.  It is useful to specify that sector
before introducing the deformation.  For ordered subsets
$I=(i_1<\cdots<i_r)$ and $J=(j_1<\cdots<j_s)$ write
\[
 Z_I:=Z_{i_1}\cdots Z_{i_r},\qquad
 Z_J^\dagger:=Z_{j_1}^\dagger\cdots Z_{j_s}^\dagger.
\]
Because $Z_i^2=(Z_j^\dagger)^2=0$, repetitions within $I$ or $J$ do not occur.
Define the represented radial PBW subspace
\begin{equation}\label{eq:represented-radial-pbw}
 \scrR^{\rm PBW}_{m,N}
 :=\operatorname{span}_{\C}
 \{Z_IZ_J^\dagger H^\alpha:\ I,J\text{ ordered},\
   \alpha\in\N^{N\times N}\}
 \subset\scrP^{\rm Cl}_{m,N}.
\end{equation}
It is the coordinate image of the standard Hermitian radial PBW module.
Faithfulness on this sector is a standard consequence of the radial-algebra
representation theorem: after complexification, a radial algebra generated
by finitely many vector variables has a faithful Clifford-polynomial
realization once the ambient vector dimension dominates the number of
abstract generators \cite[Theorem~2.1]{Sommen1997}.  In the Hermitian
realization one passes to the underlying $2m$-dimensional real Clifford space
with the $2N$ generators belonging to the $N$ complex-structure pairs; see
\cite{DeSchepperGuzmanSommen2017,DeSchepperGuzmanSommen2018Hermitian}.
Thus $m\ge N$ already suffices for this basic faithfulness statement, while
the stronger hypothesis $m\ge2N$ used here is imposed by the mixed-harmonic
stable range and in particular guarantees PBW independence on the represented
sector.  Let $\scrR_N^{\rm abs,PBW}$ denote the abstract Hermitian radial PBW
module with basis symbols
$\mathbf Z_I\mathbf Z_J^\dagger\mathbf H^\alpha$ in the same standard
order, modulo only the Hermitian radial relations of
\cite{DeSchepperGuzmanSommen2017}; equivalently, it is the abstract PBW
module whose represented basis is displayed in
\eqref{eq:represented-radial-pbw}.  We denote the generator-wise realization
by
\begin{equation}\label{eq:iota-bosonic}
 \iota_{m,N}:\scrR_N^{\rm abs,PBW}\longrightarrow
 \scrR^{\rm PBW}_{m,N},
 \qquad
 \mathbf Z_i\mapsto Z_i,\quad
 \mathbf Z_i^\dagger\mapsto Z_i^\dagger,\quad
 \mathbf H_{ij}\mapsto H_{ij}.
\end{equation}
The finite fermionic factors are adjoined in Section~\ref{sec:fermionic};
Section~\ref{sec:coordinate} then extends \eqref{eq:iota-bosonic} to the full
finite-support radial PBW module.

A \emph{coordinate extension in the localized transported setting} therefore
means an operator family defined on every element of
$\scrP^{\rm Cl}_{m,N}$ tensored with the factorwise localized fermionic charge
modules, preserving that polynomial module, agreeing with the prescribed
radial operators under the realization \eqref{eq:iota-bosonic}, and having the
required flatness, covariance, trace, and classical limit.  No finite local
differential formula is included in this definition.  The distinction is
important because Section~\ref{sec:obstructions} proves that two broad
undeformed locality classes cannot realize the exact radial restriction.

\subsection{Stable mixed-harmonic separation}
For the scalar polynomial algebra define mixed contractions and their joint
harmonic space by
\begin{equation}\label{eq:mixed-harmonics}
 \Delta_{ij}:=\sum_{a=1}^m
 \partial_{z_{ia}}\partial_{\bar z_{ja}},
 \qquad
 \scrH_{m,N}^{\rm mix}:=\bigcap_{i,j}\Ker\Delta_{ij}.
\end{equation}
Let $\scrI_N:=\C[H_{ij}:1\le i,j\le N]$.

\begin{theorem}[Stable mixed-harmonic separation]\label{thm:stable-separation}
Assume $m\ge2N$.  Multiplication is a linear isomorphism
\begin{equation}\label{eq:stable-separation}
 \scrI_N\otimes\scrH_{m,N}^{\rm mix}
 \longrightarrow\scrP^{\rm sc}_{m,N},
 \qquad I\otimes h\longmapsto Ih.
\end{equation}
Hence every scalar polynomial has a unique finite normal form
\begin{equation}\label{eq:normal-form}
 P=\sum_{\alpha\in\N^{N\times N}}H^\alpha h_\alpha(P),
 \qquad h_\alpha(P)\in\scrH_{m,N}^{\rm mix}.
\end{equation}
The normal form is $U(m)$-equivariant, label-permutation covariant, and
compatible with inclusions of label support that remain in the stable range.
\end{theorem}

\begin{proof}
Identify the scalar polynomial space with
\[
 \mathcal P\bigl(\C^m\otimes\C^N\oplus
 (\C^m)^*\otimes\C^N\bigr).
\]
For the compact dual pair $(U(m),U(N,N))$, the separation-of-variables
stable-range description states that, for $m\ge p+q$, the polynomial
oscillator representation is free over the symmetric algebra of the
holomorphic nilradical and its free factor is the joint harmonic space; see
\cite[Section~3]{HoweTanWillenbring2005} and the classical-invariant-theory
formulation in
\cite{Howe1989ClassicalInvariantTheory,HoweKimLee2017StandardMonomial}.
Here $p=q=N$.  The holomorphic nilradical is
$\C^N\otimes(\C^N)^*$; its symmetric algebra is exactly the polynomial
algebra generated by the $H_{ij}$, while the opposite nilradical acts by the
mixed contractions $\Delta_{ij}$.  Hence the cited separation theorem gives
precisely the isomorphism \eqref{eq:stable-separation}.  Uniqueness of the
factorization implies \eqref{eq:normal-form}.  The $U(m)$ action and label
permutations preserve both the invariant multipliers and the joint kernel, so
the normal form is equivariant and compatible with stable support inclusions.
\end{proof}

For the normal form \eqref{eq:normal-form} define intrinsic Gram removal and
occupancy operators by
\begin{equation}\label{eq:gram-removal}
 \mathscr R_{ij}(H^\alpha h):=\alpha_{ij}H^{\alpha-e_{ij}}h,
 \qquad
 \mathscr N_{ij}:=H_{ij}\mathscr R_{ij}.
\end{equation}
They commute pairwise and preserve polynomials.  This consequence of
uniqueness is important: ambient first-order Capelli realizations need not
commute off the invariant algebra, whereas the intrinsic normal-form
occupancies do.

\subsection{The scalar divided-power gauge}
On the radial Gram algebra the flat $q$-Hermitian theory uses the diagonal
weight
\begin{equation}\label{eq:radial-gauge}
 g_\alpha
 =\prod_{i=1}^N\frac{\alpha_{ii}!}{c_{\alpha_{ii}}!}
  \prod_{i\ne j}\frac{\alpha_{ij}!}{[\alpha_{ij}]_q!},
 \qquad
 c_r!:=\prod_{s=1}^r c_s.
\end{equation}
The stable normal form extends this weight canonically to all scalar
polynomials:
\begin{equation}\label{eq:scalar-gauge}
 \mathbb G_{m,N}^{\rm nf}(H^\alpha h):=g_\alpha H^\alpha h.
\end{equation}
Equivalently,
\begin{equation}\label{eq:gauge-factorization}
 \mathbb G_{m,N}^{\rm nf}
 =\prod_i\omega_{\rm d}(\mathscr N_{ii})
  \prod_{i\ne j}\omega_{\rm c}(\mathscr N_{ij}),
 \quad
 \omega_{\rm d}(r)=\frac{r!}{c_r!},\quad
 \omega_{\rm c}(r)=\frac{r!}{[r]_q!}.
\end{equation}

\begin{proposition}[Stable scalar lift]\label{prop:scalar-lift}
For $m\ge2N$, $\mathbb G_{m,N}^{\rm nf}$ is a degree-preserving polynomial
automorphism.  The conjugated coordinate derivatives
\[
 \partial_{z_{ka}}^{(q)}
 :=\mathbb G_{m,N}^{\rm nf}\partial_{z_{ka}}
   (\mathbb G_{m,N}^{\rm nf})^{-1},
 \qquad
 \partial_{\bar z_{ka}}^{(q)}
 :=\mathbb G_{m,N}^{\rm nf}\partial_{\bar z_{ka}}
   (\mathbb G_{m,N}^{\rm nf})^{-1}
\]
preserve scalar polynomials.  On Gram polynomials the induced Hermitian
operators are exactly \eqref{eq:radial-target-minus}--
\eqref{eq:radial-target-plus}.  The construction is $U(m)$- and
label-covariant and tends coefficientwise to the identity transport as
$q\to1$.
\end{proposition}

\begin{proof}
Theorem~\ref{thm:stable-separation} makes \eqref{eq:scalar-gauge} a well-defined
diagonal automorphism on every finite degree.  If a coordinate derivative
removes a diagonal factor $H_{kk}^r$, the neighboring gauge ratio is
$c_r/r$; if it removes a crossed factor $H_{kj}^r$, the ratio is
$[r]_q/r$.  After summing the coordinate derivatives against the Witt
symbols, these ratios give precisely the diagonal and crossed Jackson
coefficients in \eqref{eq:radial-target-minus}.  The conjugate formula is
identical.  All other assertions follow from uniqueness of the normal form
and $g_\alpha\to1$ coefficientwise.
\end{proof}

The gauge transports the scalar product,
\begin{equation}\label{eq:scalar-product}
 F\star_qG:=\mathbb G_{m,N}^{\rm nf}
 \bigl((\mathbb G_{m,N}^{\rm nf})^{-1}F\,
       (\mathbb G_{m,N}^{\rm nf})^{-1}G\bigr).
\end{equation}
This product is associative, commutative and unital; the conjugated scalar
coordinate derivatives are derivations for it.  We shall later see that this
common scalar product does \emph{not} extend to one common full
Clifford--Weyl product for both Hermitian polarizations.

\subsection{Trace coupling}
Let $n$ be the finite fermionic support and set $Q=q^{2m-2n}$.  Here and below, the \emph{scalar boundary trace} means the scalar
(identity) coefficient in the action of the Hermitian operator on the
corresponding linear supervector, after the dimension--Cartan boundary is
split into its scalar and Cartan parts.  It is not the trace of an operator
on the infinite-dimensional polynomial module.  The elementary identity
\begin{equation}\label{eq:q-subtraction}
 [2m-2n]_q=[2m]_q-Q[2n]_q
\end{equation}
has the following consequence.  If a coordinate boundary trace split as an
independently deformed bosonic contribution $[2m]_q$ plus a fermionic
contribution depending only on $n$, then it could not equal
$[2m-2n]_q$ for all $m\ge n>0$.  The fermionic scalar contraction is therefore
forced to contain the mixed parameter $Q$.  Moreover a scalar rescaling of
the whole fermionic Hermitian pair would rescale its Cartan term together
with its scalar trace, whereas the flat radial boundary deforms the dimension
coefficient but leaves the classical Cartan coefficient unchanged.  The next
section constructs a seed in which these two requirements are separated.

\section{Fermionic seed complexes and minimal localization}\label{sec:fermionic}

The coordinate transport requires more than the linear fermionic boundary:
it needs an explicit contractible odd seed on every Weyl charge, together with
the even radial defect that couples that seed to the bosonic divided-power
transport.  We record those data in a form sufficient for the later coordinate
proof.  The seed relations needed later are displayed and verified here, so
that the coordinate construction does not depend on an external seed
calculation.

\subsection{One pair and the forced denominator}
Let $\vartheta,\bar\vartheta$ be Grassmann generators with left Berezin
derivatives.  Set
\[
 N_\vartheta=\vartheta\partial_\vartheta,
 \qquad N_{\bar\vartheta}=\bar\vartheta\partial_{\bar\vartheta},
\]
and let $P_\vartheta^\epsilon,P_{\bar\vartheta}^\epsilon$ be the corresponding
occupancy projectors.  Let $a,b$ generate the Weyl algebra
\begin{equation}\label{eq:weyl}
 ab-ba=-2i,
 \qquad B_{\rm f}:=ba-i,
\end{equation}
so that
\begin{equation}\label{eq:cartan-weights}
 [B_{\rm f},a]=2ia,
 \qquad [B_{\rm f},b]=-2ib.
\end{equation}
Put
\[
 Z_{\rm f}:=\vartheta a,
 \qquad Z_{\rm f}^\dagger:=\bar\vartheta b,
 \qquad \rho_{\rm f}:=\frac1{2i}\vartheta\bar\vartheta.
\]
For a scalar seed parameter $s$ define
\begin{equation}\label{eq:PhiPsi}
 \Phi_s(B):=\frac{4-s+2iB}{B-i},
 \qquad
 \Psi_s(B):=\frac{-s-2iB}{B-i}.
\end{equation}
The Weyl shifts are
\begin{equation}\label{eq:weyl-shifts}
 bR(B_{\rm f})=R(B_{\rm f}+2i)b,
 \qquad
 aR(B_{\rm f})=R(B_{\rm f}-2i)a.
\end{equation}
On the localization adjoining $(B_{\rm f}-i)^{-1}$ define
\begin{align}
 \mathcal F_s
 &:=b\bigl(\Phi_s(B_{\rm f})P_{\bar\vartheta}^0
       +i[2]_qP_{\bar\vartheta}^1\bigr)\partial_\vartheta,
 \label{eq:F-one}\\
 \mathcal F_s^\dagger
 &:=\bigl(\Psi_s(B_{\rm f})aP_\vartheta^0
       -i[2]_qaP_\vartheta^1\bigr)\partial_{\bar\vartheta}.
 \label{eq:Fdag-one}
\end{align}

\begin{proposition}[One-pair trace--Cartan separation]\label{prop:one-pair}
The operators \eqref{eq:F-one}--\eqref{eq:Fdag-one} are square-zero and
satisfy
\begin{align}
 \mathcal F_s(Z_{\rm f})&=-s+2iB_{\rm f},
 &\mathcal F_s^\dagger(Z_{\rm f}^\dagger)&=-s-2iB_{\rm f},
 \label{eq:one-boundary}\\
 \mathcal F_s(\rho_{\rm f})&=c_1Z_{\rm f}^\dagger,
 &\mathcal F_s^\dagger(\rho_{\rm f})&=c_1Z_{\rm f}.
 \label{eq:one-rho}
\end{align}
For the one-pair Hermitian specialization used below,
$s=q^{2m-2}[2]_q$; these are the required one-pair boundary relations.  At $q=1$ the pair tends to
$2ib\partial_\vartheta$ and $-2ia\partial_{\bar\vartheta}$.
\end{proposition}

\begin{proof}
The occupancy projectors commute with the opposite Berezin derivative, hence
the squares vanish.  From \eqref{eq:weyl-shifts} and $ba=B_{\rm f}+i$,
\[
 \mathcal F_s(Z_{\rm f})
 =\Phi_s(B_{\rm f}+2i)(B_{\rm f}+i)
 =-s+2iB_{\rm f}.
\]
The conjugate identity uses $ab=B_{\rm f}-i$.  On $\rho_{\rm f}$ only the
occupied projectors contribute, and the two Berezin signs give
\eqref{eq:one-rho}.  The classical limit follows from
$\Phi_2(B)=2i$ and $\Psi_2(B)=-2i$.
\end{proof}

The denominator is not an artifact of the displayed ansatz.

\begin{theorem}[Polynomial seed obstruction]\label{thm:seed-no-go}
Consider a charge-preserving fermionic operator which is first order in
$\partial_\vartheta$, preserves the two opposite-occupancy sectors, and has a
polynomial Weyl coefficient.  On the empty opposite-occupancy sector its
coefficient has the unique right-normal form $bF(B_{\rm f})$ with
$F\in\C[B]$.  Requiring the boundary value
$-s+2iB_{\rm f}$ forces
\[
 F(B)=\frac{4-s+2iB}{B-i}.
\]
Hence no polynomial seed in this class has the required boundary when
$s\ne2$.
\end{theorem}

\begin{proof}
The degree $-1$ part of the polynomial Weyl algebra is
$b\C[B_{\rm f}]$ by PBW normal form.  The boundary equation is
$F(B+2i)(B+i)=-s+2iB$.  Shifting $B$ yields the displayed quotient.  Divisibility
by $B-i$ would require $4-s+2iB$ to vanish at $B=i$, i.e. $s=2$.
\end{proof}

\subsection{All-charge four-slot normal form and contracting seed}
For $j\in\mathbb Z$ put
\[
 x_j:=\begin{cases}a^j,&j\ge0,\\ b^{-j},&j<0,\end{cases}
\]
and define the right-normal multiplication factors
\begin{align}
 \beta_j(B)&:=\begin{cases}B+(2j-1)i,&j\ge1,\\1,&j\le0,\end{cases}
 &bx_j&=x_{j-1}\beta_j(B_{\rm f}),
 \label{eq:beta-charge-journal}\\
 \alpha_j(B)&:=\begin{cases}1,&j\ge0,\\B+(2j+1)i,&j\le-1,\end{cases}
 &ax_j&=x_{j+1}\alpha_j(B_{\rm f}).
 \label{eq:alpha-charge-journal}
\end{align}
The charge-$c$ component is the free right Cartan module with basis
\begin{equation}\label{eq:charge-basis}
 e_{0,c}:=x_c,
 \quad
 e_{\vartheta,c}:=\vartheta x_{c+1},
 \quad
 e_{\bar\vartheta,c}:=\bar\vartheta x_{c-1},
 \quad
 e_{2,c}:=\vartheta\bar\vartheta x_c.
\end{equation}
In this basis the global one-pair boundary operators preserve every charge and
have only four nonzero arrows:
\begin{align}
 \mathcal F_s(e_{\vartheta,c}R)
 &=e_{0,c}\,\sigma_c^+(B_{\rm f})R,
 &\sigma_c^+(B)&:=\beta_{c+1}(B)\Phi_s(B+2i(c+1)),
 \label{eq:sigma-plus-journal}\\
 \mathcal F_s(e_{2,c}R)
 &=e_{\bar\vartheta,c}\,\tau_c^+(B_{\rm f})R,
 &\tau_c^+(B)&:=i[2]_q\beta_c(B),
 \label{eq:tau-plus-journal}\\
 \mathcal F_s^\dagger(e_{\bar\vartheta,c}R)
 &=e_{0,c}\,\sigma_c^-(B_{\rm f})R,
 &\sigma_c^-(B)&:=\alpha_{c-1}(B)\Psi_s(B+2ic),
 \label{eq:sigma-minus-journal}\\
 \mathcal F_s^\dagger(e_{2,c}R)
 &=e_{\vartheta,c}\,\tau_c^-(B_{\rm f})R,
 &\tau_c^-(B)&:=i[2]_q\alpha_c(B).
 \label{eq:tau-minus-journal}
\end{align}
They annihilate the other two basis slots.  These formulas follow directly
from the Berezin derivative, the occupancy projectors and
\eqref{eq:beta-charge-journal}--\eqref{eq:alpha-charge-journal}.

The triangular connection uses a normalized contractible version of this
four-slot boundary.  Put
\begin{equation}\label{eq:pi-def}
 \pi_c^+:=-\frac14\sigma_c^+,
 \qquad
 \pi_c^-:=-\frac14\sigma_c^-,
\end{equation}
and define the plus seed differential by
\begin{align}
 \mathbb Q_{c;s}^+(e_{\vartheta,c}R)
 &=e_{0,c}\pi_c^+(B_{\rm f})R,
 \label{eq:Qplus-first}\\
 \mathbb Q_{c;s}^+(e_{2,c}R)
 &=\frac1{2i}e_{\bar\vartheta,c}\beta_c(B_{\rm f})R,
 \label{eq:Qplus-second}
\end{align}
with zero action on $e_{0,c}$ and $e_{\bar\vartheta,c}$.  Define the minus seed
by order-preserving Hermitian conjugation.  On any coefficient ring in which
the two displayed arrow coefficients are units, the canonical plus homotopy is
\begin{align}
 \mathbb h_{c;s}^+(e_{0,c}R)
 &=e_{\vartheta,c}(\pi_c^+(B_{\rm f}))^{-1}R,
 \label{eq:hplus-first}\\
 \mathbb h_{c;s}^+(e_{\bar\vartheta,c}R)
 &=2i\,e_{2,c}(\beta_c(B_{\rm f}))^{-1}R,
 \label{eq:hplus-second}
\end{align}
and it vanishes on $e_{\vartheta,c},e_{2,c}$.  Again the minus homotopy is the
Hermitian conjugate.  Hence, on every charge block,
\begin{equation}\label{eq:Qh-contract}
 (\mathbb Q_{c;s}^\pm)^2=(\mathbb h_{c;s}^\pm)^2=0,
 \qquad
 \{\mathbb Q_{c;s}^\pm,\mathbb h_{c;s}^\pm\}=I.
\end{equation}
Thus no abstract acyclicity assumption is being made: the homotopies are
literally the inverse arrows of two disjoint two-term complexes.

\subsection{Minimal all-charge localization}
Let $\sigma R(B):=R(B+2i)$ and set
\begin{equation}\label{eq:orbit-data}
 d(B):=B-i,
 \qquad \ell_+(B):=s-2iB,
 \qquad \ell_-(B):=s+2iB.
\end{equation}
Let $S_s$ be the multiplicative set generated by
\begin{equation}\label{eq:orbit-set}
 \{\sigma^kd,\ \sigma^k\ell_+,\ \sigma^k\ell_-:k\in\mathbb Z\},
 \qquad
 \mathcal A_s^{\min}:=S_s^{-1}\C[B_{\rm f}].
\end{equation}
Shift stability makes this a legitimate right Cartan coefficient localization
for the Weyl normal-ordering rules \eqref{eq:weyl-shifts}.

For later reference define
\begin{equation}\label{eq:Lcpm}
 L_c^+:=s-2iB_{\rm f}+4c=\sigma^c\ell_+(B_{\rm f}),
 \qquad
 L_c^-:=s+2iB_{\rm f}-4c=\sigma^c\ell_-(B_{\rm f}).
\end{equation}
With $\mathfrak a_c:=\alpha_c$ and $\mathfrak b_c:=\beta_c$ one has
\begin{align}
 \mathfrak a_c
 &=\begin{cases}1,&c\ge0,\\ \sigma^{c+1}d,&c\le-1,\end{cases}
 &
 \mathfrak b_c
 &=\begin{cases}\sigma^cd,&c\ge1,\\1,&c\le0,\end{cases}
 \label{eq:ab-factorization-journal}\\
 \pi_c^+&=\frac{L_c^+}{4\mathfrak a_c},
 &\pi_c^-&=\frac{L_c^-}{4\mathfrak b_c}.
 \label{eq:pi-factorization-journal}
\end{align}
These factorizations identify exactly which Cartan factors the inverse arrows
require.

\begin{theorem}[Minimal factorwise Cartan localization]\label{thm:min-local}
The localization $\mathcal A_s^{\min}$ is sufficient for the two conjugate
chargewise seed differentials, their canonical contracting homotopies and the
two triangular connection gauges for every $c\in\mathbb Z$.  It is minimal in
the following shift-stable sense: if a $\sigma$-stable multiplicative set
makes the canonical four-slot complexes of both polarizations contractible on
every charge, then its saturation contains $S_s$.

For finite fermionic support the minimal localization is the tensor product of
the corresponding one-pair localizations.  No mixed Cartan denominator is
needed.  At special $s$ some of the three orbit families may coalesce.
\end{theorem}

\begin{proof}
Equations \eqref{eq:ab-factorization-journal}--
\eqref{eq:pi-factorization-journal} prove sufficiency: inverse second arrows
use translates of $d^{-1}$, while inverse first arrows use in addition
translates of $\ell_+^{-1}$ and $\ell_-^{-1}$.  Conversely, on the plus complex
at charge $c=1$ the occupied second arrow has coefficient proportional to
$\sigma d$, so a contracting inverse forces $\sigma d$ to be a unit and shift
stability forces every $\sigma^kd$.  At charge zero the first plus and minus
arrows are $\ell_+/4$ and $\ell_-/4$; their inverse arrows force both orbit
families.  Tensor factorization proves the finite-support statement.
\end{proof}

\subsection{Finite fermionic support and occupancy branches}
Fix $n\ge1$.  Pair-permutation covariance forces the same one-pair seed
parameter in each fermionic factor.  The total scalar boundary condition then
forces
\[
 n s=Q_{m|n}[2n]_q,
\]
so the common value is not an additional choice.  Put
\begin{equation}\label{eq:s-symmetric}
 Q_{m|n}:=q^{2m-2n},
 \qquad
 s_{m|n}^{\rm sym}:=\frac1nQ_{m|n}[2n]_q.
\end{equation}
For charges $\mathbf c=(c_1,\ldots,c_n)$ take the graded tensor product of the
one-pair charge modules with the common value $s=s_{m|n}^{\rm sym}$.  With
hats denoting the canonical graded embeddings, set
\begin{equation}\label{eq:finite-seed}
 \mathbb Q_{\mathbf c}^{\pm,(n)}
 :=\sum_{j=1}^n\widehat{\mathbb Q}_{j,\mathbf c}^\pm,
 \qquad
 \mathbb h_{\mathbf c}^{\pm,(n)}
 :=\frac1n\sum_{j=1}^n\widehat{\mathbb h}_{j,\mathbf c}^\pm.
\end{equation}

\begin{proposition}[Finite-support contractible seed]\label{prop:finite-seed-contract}
For every $\mathbf c\in\mathbb Z^n$,
\begin{equation}\label{eq:finite-seed-identities}
 (\mathbb Q_{\mathbf c}^{\pm,(n)})^2
 =(\mathbb h_{\mathbf c}^{\pm,(n)})^2=0,
 \qquad
 \{\mathbb Q_{\mathbf c}^{\pm,(n)},
    \mathbb h_{\mathbf c}^{\pm,(n)}\}=I.
\end{equation}
The construction is $\mathfrak S_n$-covariant and Hermitian conjugation
exchanges the polarizations and $\mathbf c$ with $-\mathbf c$.  Its total
fermionic boundary is
\begin{equation}\label{eq:finite-boundary}
 -Q_{m|n}[2n]_q\ \pm\ 2i\sum_{j=1}^nB_{{\rm f},j}.
\end{equation}
Hence after adding the bosonic boundary the total scalar trace is
$[2m-2n]_q$ and the coefficient of the classical total Cartan generator is
undeformed.
\end{proposition}

\begin{proof}
Odd maps in different tensor factors anticommute, so the squares vanish and
\[
 \{\mathbb Q_{\mathbf c}^{\pm,(n)},
   \mathbb h_{\mathbf c}^{\pm,(n)}\}
 =\frac1n\sum_{j=1}^n
  \{\widehat{\mathbb Q}_{j,\mathbf c}^\pm,
    \widehat{\mathbb h}_{j,\mathbf c}^\pm\}=I.
\]
All local parameters are equal, giving permutation covariance.  The scalar
part of \eqref{eq:finite-boundary} is
$-ns_{m|n}^{\rm sym}=-Q_{m|n}[2n]_q$; together with
\eqref{eq:q-subtraction} this gives the stated trace.  Hermitian conjugacy and
the Cartan coefficient are factorwise.
\end{proof}

For the plus polarization let $\widehat{\mathsf P}_j^+$ select the local paired
slots $(e_{\bar\vartheta,c_j},e_{2,c_j})$ and vanish on
$(e_{0,c_j},e_{\vartheta,c_j})$.  For $S\subseteq\{1,\ldots,n\}$ put
\begin{equation}\label{eq:branch-projector}
 \mathsf P_S^+
 :=\prod_{j\in S}\widehat{\mathsf P}_j^+
   \prod_{j\notin S}(I-\widehat{\mathsf P}_j^+).
\end{equation}
The $2^n$ idempotents are orthogonal and sum to the identity.  A diagonal
operator commutes with both $\mathbb Q_{\mathbf c}^{+,(n)}$ and
$\mathbb h_{\mathbf c}^{+,(n)}$ exactly when it is constant on each occupancy
branch:
\begin{equation}\label{eq:boolean-commutant}
 \mathbb X=\sum_{S\subseteq\{1,\ldots,n\}}\mathsf P_S^+X_S.
\end{equation}
Indeed each local factor is the direct sum of two invertible two-term
complexes, and commuting with both an arrow and its inverse forces equality of
the two coefficients on that local pair.

To determine the coefficient $X_S$ required by radial covariance, define for
one charge
\begin{equation}\label{eq:kappa-K}
 \kappa_c^+(B):=-\frac14\sigma_c^+(B)\alpha_c(B),
 \qquad
 K_c(B):=\frac1{2i}\alpha_{c-1}(B)\beta_c(B).
\end{equation}
For $k\ge1$ define the opposite-Euler spectral factor
\begin{equation}\label{eq:Bqk}
 B_q^{(k)}(E_{\bar z})\big|_{\text{radial degree }r}
 :=\frac{c_{r+k}}{r+k}.
\end{equation}
Write $\kappa_j^+=\kappa_{c_j}^+$ and $K_j=K_{c_j}$.  For nonempty
$S$, $k=|S|$, put
\begin{equation}\label{eq:chiS}
 \chi_S^{(k)}
 :=\sum_{j\in S}\kappa_j^+-k
   -B_q^{(k)}(E_{\bar z})\sum_{j\in S}K_j,
\end{equation}
and
\begin{equation}\label{eq:M-explicit}
 \delta_m(q):=\frac{2m-[2m]_q}{4},
 \qquad
 \boxed{\mathbb M_{\mathbf c}^{+,(n)}
 :=\delta_m(q)I
   +\sum_{\varnothing\ne S\subseteq\{1,\ldots,n\}}
     \mathsf P_S^+\chi_S^{(|S|)}.}
\end{equation}
The minus defect is its order-preserving Hermitian conjugate.

\begin{proposition}[Branch coefficient and rigidity]\label{prop:branch-rigidity}
On a scalar input in branch $S$ of size $k$ and bosonic radial degree $r$,
total radial covariance forces the defect coefficient
\begin{equation}\label{eq:required-branch-defect}
 \delta_m(q)+\sum_{j\in S}\kappa_j^+-k
 -\frac{c_{r+k}}{r+k}\sum_{j\in S}K_j.
\end{equation}
Consequently \eqref{eq:M-explicit} is the unique diagonal common-commutant
defect compatible with every scalar occupancy branch.  By
\eqref{eq:boolean-commutant}, the same coefficient is then forced on every
separated Clifford channel in that branch.
\end{proposition}

\begin{proof}
Let the total radial exponent be $a=r+k$.  In the expansion of the scalar
branch with $k$ occupied nilpotent factors, the isolated bosonic vector input
has falling-factorial coefficient $(a)_k$.  The total boundary initially
contains the sum of all local $\kappa_j^+$.  For each $j\notin S$, the crossed
input with the local fermionic vector contributes the same
$(a)_k\kappa_j^+$ to the target scalar channel; subtracting those crossed
terms leaves $\sum_{j\in S}\kappa_j^+$.

In the scalar part of the total two-vector, only the contractions $K_j$ with
$j\in S$ contribute.  Each occurs with coefficient $(a-1)_{k-1}$, so division
by the isolated coefficient gives
$(a-1)_{k-1}/(a)_k=1/a=1/(r+k)$.  The divided-power lowering on augmented
degree $r+k$ supplies $c_{r+k}$.  Relative to the classical bosonic boundary,
the scalar difference is $\delta_m(q)$, and the augmented occupancy degree
contributes $-k$.  The resulting coefficient is exactly
\eqref{eq:required-branch-defect}.  Branch constancy then proves uniqueness on
the whole occupancy component.
\end{proof}

Define the aggregate triangular seed
\begin{equation}\label{eq:aggregate-N}
 \mathbb N_{\mathbf c}^{+,(n)}
 :=-\mathbb h_{\mathbf c}^{+,(n)}\mathbb M_{\mathbf c}^{+,(n)},
\end{equation}
and define the minus seed by conjugation.

\begin{proposition}[Finite-support seed relations]\label{prop:finite-seed}
For every finite support and every charge vector,
\begin{equation}\label{eq:aggregate-N-rel}
 (\mathbb N_{\mathbf c}^{\pm,(n)})^2=0,
 \qquad
 \{\mathbb Q_{\mathbf c}^{\pm,(n)},
   \mathbb N_{\mathbf c}^{\pm,(n)}\}
 =-\mathbb M_{\mathbf c}^{\pm,(n)},
 \qquad
 [\mathbb M_{\mathbf c}^{\pm,(n)},
  \mathbb N_{\mathbf c}^{\pm,(n)}]=0.
\end{equation}
No additional scalar function or separated Clifford-channel matrix is needed
at higher fermionic occupancy.
\end{proposition}

\begin{proof}
The defect \eqref{eq:M-explicit} belongs to the common diagonal commutant, so it
commutes with the differential and its contracting homotopy.  Combine this
with \eqref{eq:finite-seed-identities} and
$\mathbb N=-\mathbb h\mathbb M$.  The square vanishes because
$\mathbb h^2=0$ and $[\mathbb h,\mathbb M]=0$; the anticommutator is
$-\{\mathbb Q,\mathbb h\}\mathbb M=-\mathbb M$.
\end{proof}

\subsection{$q$-dependent oscillator nonresonance}
The abstract theory is formulated over the Ore localization above, so it does
not require every inverted factor to be nonzero in a particular oscillator
representation.  Nevertheless the $q$-dependent factors are
regular in the natural nonnegative-superdimension range.

\begin{corollary}[$q$-dependent nonresonance]\label{cor:nonresonance}
Assume $m\ge n$.  Then
\[
 0<s_{m|n}^{\rm sym}<2.
\]
On the standard polynomial oscillator
$a=\partial_u$, $b=-2iu$,
$B_{\rm f}u^r=-i(2r+1)u^r$, every shifted $q$-dependent Cartan factor is
nonzero:
\begin{align}
 (\sigma^k\ell_+)(B_{\rm f})u^r
 &=\bigl(s-2+4(k-r)\bigr)u^r\ne0,\\
 (\sigma^k\ell_-)(B_{\rm f})u^r
 &=\bigl(s+2+4(r-k)\bigr)u^r\ne0.
\end{align}
The shifted Weyl factor satisfies
\[
 (\sigma^kd)(B_{\rm f})u^r=-2i(r+1-k)u^r,
\]
so any residual all-charge oscillator resonance belongs only to this
$q$-independent orbit.  In particular, when
$m\ge\max\{2N,n\}$ the $q$-dependent part of the localization in the stable
coordinate theory is nonresonant on every polynomial degree.
\end{corollary}

\begin{proof}
Since $0<q<1$, $[2n]_q<2n$, and $m\ge n$ gives
$0<q^{2m-2n}\le1$.  Hence $0<s_{m|n}^{\rm sym}<2$.  Substituting the oscillator
Cartan eigenvalue into the shifted factors yields the displayed formulas.  A
zero of either $q$-dependent coefficient would force
$s\in\{\ldots,-6,-2,2,6,\ldots\}$, disjoint from $(0,2)$.
\end{proof}

\section{Capelli--PBW lift and triangular transport}\label{sec:capelli}

The scalar gauge of Section~\ref{sec:scalar} does not by itself act correctly
on a Clifford PBW word, because the Chevalley symbol of a Clifford product
contains lower exterior grades with additional Gram factors.  We first
construct the normalized label contractions and then use them to insert the
finite Wick correction.

\subsection{Normalized label contractions}
Let $\delta_a$ be the odd left contraction of the holomorphic Witt symbol
algebra, $\delta_a(f_b)=\delta_{ab}$ and
$\delta_a(f_b^\dagger)=0$.  Define
\begin{equation}\label{eq:s-column}
 \mathfrak s_i:=\sum_{a=1}^m\delta_a\partial_{z_{ia}},
 \qquad
 \mathcal E^z_{ji}:=\sum_{a=1}^m z_{ja}\partial_{z_{ia}},
 \qquad
 \mathcal C_z:=mI_N+(\mathcal E^z)^{\mathsf T}.
\end{equation}

\begin{lemma}[Exact label column]\label{lem:label-column}
For every scalar polynomial $F$,
\begin{equation}\label{eq:s-Z}
 \mathfrak s_i(Z_jF)=m\delta_{ij}F+\mathcal E^z_{ji}F.
\end{equation}
In the stable range $m\ge2N$, $\mathcal C_z$ is invertible on every finite
homogeneous polynomial shell and its inverse preserves polynomials.  Thus
\begin{equation}\label{eq:n-column}
 (\mathfrak n_1^+,\ldots,\mathfrak n_N^+)^{\mathsf T}
 :=\mathcal C_z^{-1}(\mathfrak s_1,\ldots,\mathfrak s_N)^{\mathsf T}
\end{equation}
is a column of degree-lowering polynomial endomorphisms satisfying
\begin{equation}\label{eq:n-Z}
 \mathfrak n_i^+(Z_jF)=\delta_{ij}F.
\end{equation}
\end{lemma}

\begin{proof}
The derivative in \eqref{eq:s-column} either hits $Z_j$, giving
$m\delta_{ij}F$, or $F$, giving $\mathcal E^z_{ji}F$, which proves
\eqref{eq:s-Z}.  It remains to justify the inverse without introducing a
formal matrix denominator.

Fix a homogeneous degree.  The label operators $\mathcal E^z_{ji}$ generate
the natural $\mathfrak{gl}_N$ action, so the finite-dimensional homogeneous
space is a direct sum of polynomial $GL(N)$-modules $F_\lambda$.  On the
column tensor $\C^N\otimes F_\lambda$, the matrix
$(\mathcal E^z)^{\mathsf T}$ is the standard tensor-Casimir operator.  The Pieri decomposition of $\C^N\otimes F_\lambda$ is multiplicity-free,
so this tensor-Casimir is semisimple on the shell.  Its eigenvalue on the
branch obtained by adding a box in row $r$ is
\[
 \lambda_r-r+1,\qquad 1\le r\le N.
\]
(Equivalently, this is the usual Capelli shift for the defining
representation.)  Therefore every eigenvalue of
$\mathcal C_z=mI_N+(\mathcal E^z)^{\mathsf T}$ is
\[
 m+\lambda_r-r+1\ge m-N+1>0.
\]
Thus $\mathcal C_z$ is invertible on each homogeneous shell.  Since each
shell is finite dimensional, its inverse is a finite spectral polynomial in
$\mathcal C_z$ on that shell and consequently maps polynomials to
polynomials.  Multiplying the exact column relation \eqref{eq:s-Z} by this
inverse proves \eqref{eq:n-Z}.
\end{proof}

Let $\mathscr B^+$ denote the subspace annihilated by all $\delta_a$; it
contains scalars and standard words formed only from daggered Witt symbols.
For an ordered label set $I=(i_1<\cdots<i_r)$ define the abstract Grassmann
label contraction
\begin{equation}\label{eq:abstract-c}
 \mathfrak c_i^+(Z_IB)
 :=\sum_{t=1}^r(-1)^{t-1}\delta_{i,i_t}
 Z_{i_1}\cdots\widehat{Z_{i_t}}\cdots Z_{i_r}B,
 \qquad B\in\mathscr B^+.
\end{equation}

\begin{theorem}[Capelli--PBW factorization]\label{thm:capelli-pbw}
On the standard plus-polarized radial PBW module,
\begin{equation}\label{eq:capelli-factor}
 \mathfrak s_i
 =\sum_{\ell=1}^N
  (m\delta_{i\ell}+\mathcal E^z_{\ell i})\mathfrak c_\ell^+.
\end{equation}
Consequently, for $m\ge2N$,
\begin{equation}\label{eq:n-c}
 \mathfrak n_i^+=\mathfrak c_i^+
\end{equation}
on every standard radial exterior word.  The order-preserving Hermitian
conjugate gives the corresponding statement for $\mathfrak n_i^-$.
\end{theorem}

\begin{proof}
Let $L_j$ be left multiplication by $Z_j$ and
$\mathcal K^+:=\sum_\ell L_\ell\mathfrak c_\ell^+$.  The graded Leibniz rule
gives
\begin{equation}\label{eq:sL}
 \mathfrak s_iL_j+L_j\mathfrak s_i
 =m\delta_{ij}+\mathcal E^z_{ji}-\delta_{ij}\mathcal K^+.
\end{equation}
Put
$\mathfrak r_i:=\sum_\ell(m\delta_{i\ell}+\mathcal E^z_{\ell i})
\mathfrak c_\ell^+$.  Using
$[\mathcal E^z_{\ell i},L_j]=\delta_{ij}L_\ell$ and
$\mathfrak c_\ell^+L_j+L_j\mathfrak c_\ell^+=\delta_{\ell j}$ shows that
$\mathfrak r_i$ obeys the same relation \eqref{eq:sL}.  Both
$\mathfrak s_i$ and $\mathfrak r_i$ annihilate $\mathscr B^+$, hence they
coincide by induction on the number of holomorphic standard factors.
Multiplying the column identity by $\mathcal C_z^{-1}$ proves
\eqref{eq:n-c}.
\end{proof}

In particular, all finite exterior overlaps are already encoded by the
canonical label contraction.  For example,
\[
 \mathfrak n_i^+(Z_jZ_kF)
 =\delta_{ij}Z_kF-\delta_{ik}Z_jF,
 \qquad
 \mathfrak n_i^+(Z_jZ_k^\dagger F)=\delta_{ij}Z_k^\dagger F.
\]
No additional two-label curvature term is needed.

\subsection{Finite Wick correction}
Let $\sigma$ be the Chevalley symbol isomorphism from the complex Clifford
algebra to the exterior algebra on the Witt symbols.  In the exterior-symbol
model define
\begin{equation}\label{eq:Kwick}
 \mathbb K_{m,N}
 :=\frac12\sum_{i,j=1}^N
 L_{H_{ij}}\,\mathfrak n_j^-\mathfrak n_i^+,
 \qquad
 \mathbb W_{m,N}:=\sigma^{-1}\exp(\mathbb K_{m,N})\sigma.
\end{equation}
The exponential is finite because $\mathbb K_{m,N}$ lowers total exterior
degree by two.

\begin{lemma}[Wick realization]\label{lem:wick}
For ordered label sets $I,J$ and every Gram monomial $H^\alpha$,
\begin{equation}\label{eq:wick-radial}
 \sigma(Z_IZ_J^\dagger H^\alpha)
 =\exp(\mathbb K_{m,N})
  ((Z_I\wedge Z_J^\dagger)H^\alpha).
\end{equation}
\end{lemma}

\begin{proof}
By Theorem~\ref{thm:capelli-pbw} and its conjugate,
$\mathfrak n_i^+$ and $\mathfrak n_j^-$ are the canonical label contractions
on every standard radial word.  Thus $\mathbb K_{m,N}$ contracts one
holomorphic and one antiholomorphic exterior factor and inserts one half of
their scalar anticommutator $H_{ij}$.  Expanding the finite exponential sums
over disjoint mixed pairings with the Grassmann signs, which is precisely the
Clifford Wick formula.
\end{proof}

Let $\mathbb G_{m,N}^{\rm nf,\Lambda}$ be the scalar normal-form gauge acting
coefficientwise in the exterior-symbol basis and put
\begin{equation}\label{eq:full-pbw-gauge}
 \mathbb G_{m,N}^{\rm PBW}
 :=\mathbb W_{m,N}\mathbb G_{m,N}^{\rm nf,\Lambda}
   \mathbb W_{m,N}^{-1}.
\end{equation}

\begin{theorem}[Full PBW gauge]\label{thm:full-pbw}
For $m\ge2N$, $\mathbb G_{m,N}^{\rm PBW}$ is a degree-preserving polynomial
automorphism of the full Clifford-valued coordinate polynomial space and
\begin{equation}\label{eq:PBW-radial}
 \mathbb G_{m,N}^{\rm PBW}(Z_IZ_J^\dagger H^\alpha)
 =g_\alpha Z_IZ_J^\dagger H^\alpha
\end{equation}
on every represented radial PBW monomial.  It is $U(m)$- and label-covariant,
support-compatible in the stable range, and becomes the identity at $q=1$.
\end{theorem}

\begin{proof}
Polynomiality and invertibility follow from the finite Wick exponential and
the finite-degree scalar gauge.  Apply $\mathbb W_{m,N}^{-1}$ to a radial PBW
word and use Lemma~\ref{lem:wick}; the result is the pure exterior word
$(Z_I\wedge Z_J^\dagger)H^\alpha$.  The middle scalar gauge multiplies it by
$g_\alpha$, and the final Wick map reconstructs the original Clifford PBW
word.
\end{proof}

\subsection{Common triangular transport}
For each bosonic label $k$, let $\mathscr F_k$ be the finite-support
fermionic charge module of Section~\ref{sec:fermionic}.  Suppress the charge
indices and write its plus-polarized seed operators as
$\mathbb Q_k^+,\mathbb N_k^+$; the minus-polarized operators are their conjugates.  On the
graded tensor product define the ungauged flat super-Dirac seeds
\begin{equation}\label{eq:Ak}
 \mathbb A_k^+
 :=\partial_{Z_k}\widehat\otimes I+I\widehat\otimes\mathbb Q_k^+,
 \qquad
 \mathbb A_k^-:=(\mathbb A_k^+)^\ddagger.
\end{equation}
They satisfy
\begin{equation}\label{eq:A-flat}
 (\mathbb A_k^\pm)^2=0,
 \qquad
 \{\mathbb A_i^\pm,\mathbb A_j^\pm\}=0\quad(i\ne j).
\end{equation}
The notation $\mathfrak n_i^+\widehat\otimes\mathbb N_i^+$ is
operator-valued: $\mathbb N_i^+$ is the odd finite-support homotopy from
Section~\ref{sec:fermionic}, including its right-normal coefficient depending
on the opposite Euler shell.  Thus the tensor sign records the graded
fermionic action, while the displayed spectral coefficient acts on the local
bosonic Euler degree.  This convention is polynomial-preserving on the
bosonic coordinate module over the localized fermionic coefficient ring.
Set
\begin{equation}\label{eq:H-U}
 \mathbb H_q^+:=\sum_{i=1}^N
 \mathfrak n_i^+\widehat\otimes\mathbb N_i^+,
 \qquad
 \mathbb U_q^+:=\exp(\mathbb H_q^+),
\end{equation}
and define $\mathbb H_q^-,\mathbb U_q^-$ by Hermitian conjugation.  The
exponentials are finite on every polynomial--finite-support block because the
fermionic homotopies are filtration increasing and nilpotent.

\begin{proposition}[Triangular flatness]\label{prop:tri-flat}
The operators
\begin{equation}\label{eq:tri-D}
 \mathbb B_{k,q}^\pm:=(\mathbb U_q^\pm)^{-1}
  \mathbb A_k^\pm\mathbb U_q^\pm
\end{equation}
are polynomial-preserving and obey the same-polarization relations
\begin{equation}\label{eq:B-flat}
 (\mathbb B_{k,q}^\pm)^2=0,
 \qquad
 \{\mathbb B_{i,q}^\pm,\mathbb B_{j,q}^\pm\}=0\quad(i\ne j).
\end{equation}
Their linear boundary correction is label diagonal.
\end{proposition}

\begin{proof}
The finite exponential is an invertible polynomial-preserving endomorphism
of the bosonic coordinate module over the localized fermionic coefficient
ring, so \eqref{eq:B-flat} is the common conjugate of \eqref{eq:A-flat}.  Naturality of
the Capelli column and the seed construction give label covariance.  Equation
\eqref{eq:n-Z} makes the first boundary correction label diagonal; higher
filtration terms vanish on a linear input.
\end{proof}

\section{The stable coordinate families}\label{sec:coordinate}

Extend the PBW gauge trivially over the finite fermionic factors.  The total
transports are
\begin{equation}\label{eq:Tpm}
 \mathbb T_q^+:=\mathbb G_{m,N}^{\rm PBW}(\mathbb U_q^+)^{-1},
 \qquad
 \mathbb T_q^-:=\mathbb G_{m,N}^{\rm PBW}(\mathbb U_q^-)^{-1}.
\end{equation}
The order in \eqref{eq:Tpm} is important: the scalar divided-power gauge acts
outside the triangular trace correction.  Define
\begin{equation}\label{eq:Dpm}
 \mathbb D_{k,q}^\pm
 :=\mathbb T_q^\pm\mathbb A_k^\pm(\mathbb T_q^\pm)^{-1}
 =\mathbb G_{m,N}^{\rm PBW}(\mathbb U_q^\pm)^{-1}
  \mathbb A_k^\pm\mathbb U_q^\pm
  (\mathbb G_{m,N}^{\rm PBW})^{-1}.
\end{equation}

Let $\scrF_{\mathrm{loc}}^{(n)}$ be the direct sum of the finite charge
blocks of Section~\ref{sec:fermionic}, with the factorwise localization
\(\mathcal A_{s_{m|n}^{\rm sym}}^{\min}\).  Put
\[
 \scrR_{N,n}^{\rm abs,PBW}
 :=\scrR_N^{\rm abs,PBW}\widehat\otimes\scrF_{\mathrm{loc}}^{(n)}.
\]
Extend the bosonic realization \eqref{eq:iota-bosonic} by the identity on
those charge factors.  Thus
\begin{equation}\label{eq:iota-full}
 \iota_{m,n,N}:
 \scrR_{N,n}^{\rm abs,PBW}
 \longrightarrow
 \scrR_{m,N}^{\rm PBW}\widehat\otimes
 \scrF_{\mathrm{loc}}^{(n)}
 \subset
 \scrP_{m,N}^{\rm Cl}\widehat\otimes
 \scrF_{\mathrm{loc}}^{(n)}
\end{equation}
is the generator-wise standard realization of the finite-support Hermitian
radial PBW module.  The stable PBW independence discussed in
Section~\ref{sec:scalar} makes this realization faithful on the standard
radial sector used here.  On the radial side, write
$\mathfrak c_i^{\pm,\rm rad}$ for the abstract label contractions and
\begin{equation}\label{eq:Urad}
 \mathbb U_{q,\rm rad}^\pm
 :=\exp\left(\sum_{i=1}^N
 \mathfrak c_i^{\pm,\rm rad}\widehat\otimes\mathbb N_i^\pm\right).
\end{equation}
Let $\mathcal G_N$ be the radial PBW gauge acting by the same weights
$g_\alpha$ as in \eqref{eq:radial-gauge}.  Let
$\mathfrak A_{k,\mathrm{rad}}^\pm$ denote the ungauged Hermitian radial
super-Dirac seeds: the abstract bosonic Hermitian vector derivative in label
$k$ plus the (generally $q$-dependent) finite-support seed $\mathbb Q_k^\pm$ of
Section~\ref{sec:fermionic}.  Define, entirely on the abstract radial PBW
module,
\begin{equation}\label{eq:radial-D-definition}
 \mathfrak D_{k,q}^\pm
 :=\mathcal G_N(\mathbb U_{q,\rm rad}^\pm)^{-1}
   \mathfrak A_{k,\rm rad}^\pm
   \mathbb U_{q,\rm rad}^\pm\mathcal G_N^{-1}.
\end{equation}
This is the flat finite-support radial calculus used as the reference calculus in the
Introduction; its scalar columns are \eqref{eq:radial-target-minus}--
\eqref{eq:radial-target-plus}.  Formula \eqref{eq:radial-D-definition} is also
the convenient definition needed for the coordinate comparison below.

\begin{theorem}[Exact radial intertwining]\label{thm:radial-intertwining}
Assume $m\ge2N$ and finite fermionic support.  Then
\begin{equation}\label{eq:intertwine-U}
 \mathbb U_q^\pm\iota_{m,n,N}
 =\iota_{m,n,N}\mathbb U_{q,\rm rad}^\pm,
 \qquad
 \mathbb G_{m,N}^{\rm PBW}\iota_{m,n,N}
 =\iota_{m,n,N}\mathcal G_N.
\end{equation}
Consequently
\begin{equation}\label{eq:D-radial}
 \mathbb D_{k,q}^\pm\iota_{m,n,N}
 =\iota_{m,n,N}\mathfrak D_{k,q}^\pm,
\end{equation}
where $\mathfrak D_{k,q}^\pm$ are the radial operators defined in
\eqref{eq:radial-D-definition}, with $Q=q^{2m-2n}$.
\end{theorem}

\begin{proof}
Theorem~\ref{thm:capelli-pbw} identifies every coordinate
$\mathfrak n_i^\pm$ with the represented abstract contraction on the full
standard radial PBW module, so exponentiation gives the first identity in
\eqref{eq:intertwine-U}.  Theorem~\ref{thm:full-pbw} gives the second identity
on every radial PBW monomial, including all Clifford grades.  The seed maps
and their occupancy-branch coefficients are the same on the two sides by
Proposition~\ref{prop:finite-seed}.  Combining the two intertwining relations
with the order \eqref{eq:Tpm} gives \eqref{eq:D-radial}.
\end{proof}

\begin{theorem}[Stable transported coordinate calculus]\label{thm:stable-coordinate}
Let $N,n<\infty$, $n\ge1$, and assume $m\ge2N$.  On arbitrary
Clifford-valued bosonic coordinate polynomials tensored with the factorwise
localized finite-support fermionic charge modules, the families
$\mathbb D_{k,q}^\pm$ of \eqref{eq:Dpm} satisfy:
\begin{enumerate}[label=(\roman*)]
\item
\[
 (\mathbb D_{k,q}^\pm)^2=0,
 \qquad
 \{\mathbb D_{i,q}^\pm,\mathbb D_{j,q}^\pm\}=0\quad(i\ne j);
\]
\item their standard radial restriction is exactly the flat finite-support
$q$-Hermitian PBW calculus;
\item the scalar boundary trace is $[2m-2n]_q$ and the coefficient of the classical
total Cartan generator is undeformed;
\item the construction is $U(m)$-equivariant, label-permutation covariant,
fermionic-pair permutation covariant and compatible with inclusions of finite
support that remain in the stable range;
\item no inverse Gram determinant occurs: the bosonic operators preserve the
ordinary coordinate polynomial module, while the only algebraic localization
is the factorwise Cartan localization of Section~\ref{sec:fermionic};
\item coefficientwise, $\mathbb D_{k,q}^\pm$ tends to the classical Hermitian
super Dirac pair as $q\to1$.
\end{enumerate}
\end{theorem}

\begin{proof}
Same-polarization flatness follows from Proposition~\ref{prop:tri-flat} by
conjugation with the common PBW gauge.  The exact radial statement is
Theorem~\ref{thm:radial-intertwining}.

For the linear boundary, the bosonic divided-power column has scalar trace
$[2m]_q$.  Proposition~\ref{prop:finite-seed-contract} gives the fermionic
scalar contribution
$-Q_{m|n}[2n]_q$ and the undeformed Cartan contribution
$\pm2i\sum_jB_{{\rm f},j}$.  Therefore
\[
 [2m]_q-Q_{m|n}[2n]_q=[2m-2n]_q
\]
by \eqref{eq:q-subtraction}.  The Capelli--PBW gauge and triangular
conjugation change the realization of the boundary column but not this scalar
trace identity, because Theorem~\ref{thm:radial-intertwining} identifies that
column with the radial one exactly.  This proves (iii), including the
undeformed Cartan coefficient.

Equivariance and support compatibility follow from the stable normal form,
the natural Capelli column and the factorwise symmetric seed.  Polynomial
preservation follows from Lemma~\ref{lem:label-column},
Theorem~\ref{thm:full-pbw}, and finiteness of the triangular exponentials.  At
$q=1$ the divided-power weights are $1$, the scalar defect
$\delta_m(q)$ and all finite-support defect coefficients vanish, and the normalized
fermionic seed becomes the classical one.  Hence both total transports tend
coefficientwise to the identity and the classical Hermitian super Dirac pair
is recovered.
\end{proof}

\begin{remark}[Exact scope]\label{rem:scope}
Theorem~\ref{thm:stable-coordinate} is a coordinate theorem, but not a fixed
local differential realization of the 2018 type.  The bosonic coordinate
module is polynomial and no inverse Gram determinant appears, yet the
construction uses homogeneous-shell Capelli inverses and the factorwise
Cartan localization.  Section~\ref{sec:obstructions} proves that two broad
undeformed locality classes cannot reproduce the exact radial restriction for
$0<q<1$.
\end{remark}

\section{Mixed polarizations and the complete relative algebra}\label{sec:mixed}

The two polarizations share the same outer PBW gauge but use conjugate
triangular frames.  Their mixed algebra is therefore controlled by the
relative triangular transport rather than by one common conjugation.

\subsection{Relative transport}
Put
\begin{equation}\label{eq:Lcl}
 \mathbb L_{ij}^{\rm seed}:=\{\mathbb A_i^+,\mathbb A_j^-\}.
\end{equation}
By the graded tensor convention the cross terms vanish, hence explicitly
\begin{equation}\label{eq:Lcl-explicit}
 \mathbb L_{ij}^{\rm seed}
 =\Delta_{ij}\widehat\otimes I
  +I\widehat\otimes\{\mathbb Q_i^+,\mathbb Q_j^-\}.
\end{equation}
Thus the purely bosonic restriction is $\Delta_{ij}$.  The second summand is
the ungauged finite-support fermionic mixed contraction and is generally
$q$-dependent through the seed parameter; at $q=1$ it becomes the classical
fermionic mixed contraction.  Define
\begin{equation}\label{eq:R}
 \mathbb R_q:=\mathbb U_q^+(\mathbb U_q^-)^{-1},
 \qquad
 \boldsymbol\Omega_{j,q}^{-\mid+}
 :=\mathbb R_q\mathbb A_j^-\mathbb R_q^{-1}-\mathbb A_j^-.
\end{equation}
At $q=1$, $\mathbb R_1=I$ and
$\boldsymbol\Omega_{j,1}^{-\mid+}=0$.

\begin{theorem}[Relative-transport formula]\label{thm:relative}
For $m\ge2N$ and finite fermionic support,
\begin{equation}\label{eq:mixed-relative}
 \boxed{
 \{\mathbb D_{i,q}^+,\mathbb D_{j,q}^-\}
 =\mathbb G_{m,N}^{\rm PBW}(\mathbb U_q^+)^{-1}
 \left(\mathbb L_{ij}^{\rm seed}
 +\{\mathbb A_i^+,\boldsymbol\Omega_{j,q}^{-\mid+}\}\right)
 \mathbb U_q^+(\mathbb G_{m,N}^{\rm PBW})^{-1}. }
\end{equation}
Thus every additional mixed term is contained in the single finite
relative coefficient $\boldsymbol\Omega_{j,q}^{-\mid+}$.
\end{theorem}

\begin{proof}
Conjugating the actual minus frame by $\mathbb U_q^+$ gives
\[
 \mathbb U_q^+(\mathbb U_q^-)^{-1}\mathbb A_j^-
 \mathbb U_q^-(\mathbb U_q^+)^{-1}
 =\mathbb R_q\mathbb A_j^-\mathbb R_q^{-1}
 =\mathbb A_j^-+\boldsymbol\Omega_{j,q}^{-\mid+}.
\]
Taking the anticommutator with the plus seed $\mathbb A_i^+$ and restoring
the common outer PBW gauge yields \eqref{eq:mixed-relative}.
\end{proof}

\subsection{The local rank-one defect}
The factorized formulas below are global, but it is useful to record the
complete local defect once.  For one bosonic label write
\[
 E:=E_z,\qquad \bar E:=E_{\bar z},
\]
and for a spectral coefficient $F(E,\bar E)$ put
\[
 \tau_z^{-a}F:=F(E-a,\bar E),
 \qquad
 \tau_{\bar z}^{-b}F:=F(E,\bar E-b).
\]
These shifts are always placed to the right of an operator which has already
lowered the corresponding degree.

Define the normalized mixed Laplacian
\begin{equation}\label{eq:Lambda-local}
 \boldsymbol\Lambda^{+|-}:=\{\mathfrak n^+,\partial_{Z^\dagger}\}.
\end{equation}
For arbitrary labels the same computation gives
\begin{equation}\label{eq:Lambda-multilabel}
 \boldsymbol\Lambda_{ij}^{+|-}
 =\sum_{\ell=1}^N(\mathcal C_z^{-1})_{i\ell}\Delta_{\ell j}.
\end{equation}
Indeed, $\{\mathfrak s_\ell,\partial_{Z_j^\dagger}\}=\Delta_{\ell j}$ and
$\mathcal C_z$ commutes with the antiholomorphic derivative.

The Capelli contractions lower one Euler degree, so on homogeneous shells
\begin{equation}\label{eq:euler-shift-n}
 \mathfrak n^+F(E)=F(E+1)\mathfrak n^+,
 \qquad
 F(E)\mathfrak n^+=\mathfrak n^+F(E-1),
\end{equation}
with the conjugate formula for $\mathfrak n^-$ and $\bar E$.
Let
\[
 d^+:=\partial_Z,\quad d^-:=\partial_{Z^\dagger},\quad
 e^\pm:=\{d^\pm,\mathfrak n^\pm\}.
\]
For one label and one fermionic pair, write
$\mathbb Q^\pm,\mathbb M^\pm,\mathbb N^\pm$ for the four-slot local seed operators and
\begin{equation}\label{eq:Kminus-local}
 \mathbb K^-:=[\mathbb A^-,\mathbb H^-]
 =\mathfrak n^-\widehat\otimes\mathbb M^-
  +e^-\widehat\otimes\mathbb N^-.
\end{equation}

\begin{proposition}[Correct first relative commutator]\label{prop:first-local}
For arbitrary labels,
\begin{align}
 [\mathbb H_i^+,\mathbb A_j^-]
 &=-\boldsymbol\Lambda_{ij}^{+|-}
    \widehat\otimes\mathbb N_i^+
 \notag\\
 &\quad+\delta_{ij}\,
   \mathfrak n_i^+\partial_{Z_i^\dagger}
   \widehat\otimes
   (\mathbb N_i^+-\tau_{\bar i}^{-1}\mathbb N_i^+)
 \notag\\
 &\quad+\delta_{ij}\,
   \mathfrak n_i^+\widehat\otimes
   \{\mathbb N_i^+,\mathbb Q_i^-\}.
 \label{eq:first-local}
\end{align}
The middle line is the adjacent-shell correction and is essential whenever
the plus seed coefficient depends on the opposite Euler degree.
\end{proposition}

\begin{proof}
The bosonic superbracket gives the first term through
\eqref{eq:Lambda-multilabel}.  For $i=j$, moving the opposite-degree spectral
coefficient back to the right after $\partial_{Z_i^\dagger}$ uses
\eqref{eq:euler-shift-n} and produces the second line.  The fermionic part of
$\mathbb A_j^-$ gives the last line.  For $i\ne j$ the seed coefficient
commutes with the opposite derivative and only the first term remains.
\end{proof}

\begin{theorem}[Exact one-label relative connection]\label{thm:local-Omega}
On every rank-one polynomial--localized one-pair block,
\begin{align}
 \boldsymbol\Omega_q^{-|+}
 &=\mathbb K^-
   -\boldsymbol\Lambda^{+|-}\widehat\otimes\mathbb N^+
 \notag\\
 &\quad+\mathfrak n^+d^-\widehat\otimes
    (\mathbb N^+-\tau_{\bar z}^{-1}\mathbb N^+)
   +\mathfrak n^+\widehat\otimes\{\mathbb N^+,\mathbb Q^-\}
 \notag\\
 &\quad+\mathfrak n^+\mathfrak n^-\widehat\otimes
   \bigl((\tau_z^{-1}\mathbb M^-)\mathbb N^+
        -(\tau_{\bar z}^{-1}\mathbb N^+)\mathbb M^-\bigr)
 \notag\\
 &\quad+\mathfrak n^+e^-\widehat\otimes
      (\tau_{\bar z}^{-2}\mathbb N^+)\mathbb N^-
   +e^-\mathfrak n^+\widehat\otimes
      (\tau_z^{-1}\mathbb N^-)\mathbb N^+.
 \label{eq:local-Omega}
\end{align}
There is no further rank-one BCH term.
\end{theorem}

\begin{proof}
Since $(\mathbb H^\pm)^2=0$, the conjugation by $I+\mathbb H^-$
contains at most the sandwich $\mathbb H^-\mathbb A^-\mathbb H^-$.  Its
bosonic-derivative part contains $(\mathbb N^-)^2$, whereas its fermionic
part contains $(\mathfrak n^-)^2$; both vanish.  Hence
$(\mathbb U^-)^{-1}\mathbb A^-\mathbb U^-=\mathbb A^-+\mathbb K^-$.  Conjugate
this by $\mathbb U^+=I+\mathbb H^+$ and expand the finite BCH expression
\[
 \boldsymbol\Omega_q^{-|+}
 =\mathbb K^-+[\mathbb H^+,\mathbb A^-+\mathbb K^-]
 +\frac12[\mathbb H^+,[\mathbb H^+,\mathbb A^-+\mathbb K^-]].
\]
The commutator with $\mathbb A^-$ is Proposition~\ref{prop:first-local}.
For the two terms of \eqref{eq:Kminus-local}, graded multiplication and
\eqref{eq:euler-shift-n} give the last three terms of
\eqref{eq:local-Omega}.  Every seed sandwich in the second BCH commutator
contains one of
$(\mathbb N^+)^2$,
$\mathbb N^+\mathbb Q^-\mathbb N^+$,
$\mathbb N^+\mathbb M^-\mathbb N^+$, or
$\mathbb N^+\mathbb N^-\mathbb N^+$; the four-slot support of the seed makes
all of them zero.  Hence the displayed formula is complete.
\end{proof}

The preceding formula is local in the label but already displays the two
normal-ordering effects that persist for arbitrary $N$: Euler shifts of
seed coefficients and derivatives of the inverse Capelli matrix.  The latter
are controlled by the following exact identities.  Write
$d_k^+:=\partial_{Z_k}$.

\begin{proposition}[Capelli spin--Cartan identities]\label{prop:spin-cartan}
For all labels $i,k$ in the stable range,
\begin{align}
 [d_k^+,\mathcal C_z]_{i\ell}
 &=\delta_{k\ell}d_i^+,\label{eq:d-C}\\
 [d_k^+,\mathcal C_z^{-1}]_{i\ell}
 &=-\sum_{r=1}^N
  (\mathcal C_z^{-1})_{ir}d_r^+
  (\mathcal C_z^{-1})_{k\ell}.
 \label{eq:d-Cinv}
\end{align}
Consequently the even same-polarization matrix
\begin{equation}\label{eq:spin-cartan}
 \mathfrak E_{ki}^+:=\{d_k^+,\mathfrak n_i^+\}
 =-\sum_{r=1}^N
  (\mathcal C_z^{-1})_{ir}d_r^+\mathfrak n_k^+
\end{equation}
and the normalized mixed Laplacian column satisfies
\begin{equation}\label{eq:d-Lambda}
 [d_k^+,\boldsymbol\Lambda_{ij}^{+|-}]
 =-\sum_{r=1}^N
  (\mathcal C_z^{-1})_{ir}d_r^+
  \boldsymbol\Lambda_{kj}^{+|-}.
\end{equation}
The Hermitian-conjugate formulas hold in the minus polarization.
\end{proposition}

\begin{proof}
Coordinate differentiation gives
$[d_k^+,\mathcal E^z_{\ell i}]=\delta_{k\ell}d_i^+$, which is
\eqref{eq:d-C}.  The inverse commutator identity
$[D,C^{-1}]=-C^{-1}[D,C]C^{-1}$ gives \eqref{eq:d-Cinv}.  Since the
unnormalized contractions $\mathfrak s_\ell$ anticommute with every $d_k^+$,
substitution into $\mathfrak n_i^+=\sum_\ell
(\mathcal C_z^{-1})_{i\ell}\mathfrak s_\ell$ yields
\eqref{eq:spin-cartan}.  Finally $d_k^+$ commutes with each scalar
$\Delta_{\ell j}$, so the same inverse-matrix calculation applied to
\eqref{eq:Lambda-multilabel} proves \eqref{eq:d-Lambda}.
\end{proof}

Appendix~\ref{app:mixed-normal-order} applies these identities and the Euler
shift rule to \eqref{eq:local-Omega} term by term.  It gives the complete
rank-one one-pair formula for
$\{\mathbb A^+,\boldsymbol\Omega_q^{-|+}\}$; hence no normal-ordering step is
the remaining terms are handled by the factorization below.

This motivates the mixed $q$-Laplacian matrix
\begin{equation}\label{eq:Lq}
 \mathbb L_{ij}^{(q)}:=\{\mathbb D_{i,q}^+,\mathbb D_{j,q}^-\}.
\end{equation}
It is polynomial-preserving, $U(m)$- and label-covariant, obeys
$(\mathbb L_{ij}^{(q)})^\ddagger=\mathbb L_{ji}^{(q)}$, and tends, as $q\to1$, to the classical mixed bracket obtained by
specializing $\mathbb A_i^\pm$ at $q=1$.

\subsection{Global Capelli--Grassmann relations}
The Capelli contractions are canonical on radial PBW words by
Theorem~\ref{thm:capelli-pbw}; we now show that their Grassmann relations
hold on the whole coordinate module.

\begin{lemma}[Label covariance]\label{lem:n-covariance}
For $\mathcal E^z_{ab}=\sum_c z_{ac}\partial_{z_{bc}}$,
\begin{equation}\label{eq:n-covariance}
 [\mathcal E^z_{ab},\mathfrak n_i^+]
 =-\delta_{ia}\mathfrak n_b^+.
\end{equation}
The conjugate identity holds for $\mathfrak n_i^-$.
\end{lemma}

\begin{proof}
The unnormalized column satisfies
$[\mathcal E^z_{ab},\mathfrak s_i]=-
\delta_{ia}\mathfrak s_b$ and
\[
 [\mathcal E^z_{ab},\mathcal C_z]
 =\mathcal C_ze_{ab}-e_{ab}\mathcal C_z.
\]
Hence
$[\mathcal E^z_{ab},\mathcal C_z^{-1}]
=\mathcal C_z^{-1}e_{ab}-e_{ab}\mathcal C_z^{-1}$.
Apply this to $\mathfrak n^+=\mathcal C_z^{-1}\mathfrak s$.
\end{proof}

\begin{theorem}[Global Capelli--Grassmann relations]\label{thm:grassmann}
In the stable range $m\ge2N$, on arbitrary Clifford-valued coordinate
polynomials,
\begin{align}
 \{\mathfrak n_i^+,\mathfrak n_j^+\}&=0,
 &\{\mathfrak n_i^-,\mathfrak n_j^-\}&=0,
 \label{eq:grass-same}\\
 \{\mathfrak n_i^+,\mathfrak n_j^-\}&=0.
 \label{eq:grass-cross}
\end{align}
In particular $(\mathfrak n_i^\pm)^2=0$ globally.
\end{theorem}

\begin{proof}
We prove the plus relation.  Put
$F_{ab}:=\{\mathfrak n_a^+,\mathfrak n_b^+\}$.  The unnormalized
contractions satisfy $\{\mathfrak s_i,\mathfrak s_j\}=0$.  From
Lemma~\ref{lem:n-covariance},
\[
 [\mathfrak n_a^+,(\mathcal C_z)_{jb}]
 =\delta_{ab}\mathfrak n_j^+.
\]
Using
$\mathfrak s_i=\sum_a(\mathcal C_z)_{ia}\mathfrak n_a^+$ and
\[
 [(\mathcal C_z)_{ia},(\mathcal C_z)_{jb}]
 =\delta_{ib}(\mathcal C_z)_{ja}
  -\delta_{ja}(\mathcal C_z)_{ib},
\]
a direct expansion gives
\[
 \sum_a(\mathcal C_z)_{ia}
 \left(\sum_b(\mathcal C_z)_{jb}F_{ab}+F_{aj}\right)=0.
\]
Invertibility of $\mathcal C_z$ reduces this to
$(\mathcal C_z+I_N)(F_{a1},\ldots,F_{aN})^{\mathsf T}=0$ for each $a$.
The proof of Lemma~\ref{lem:label-column} applies with $m$ replaced by $m+1$,
so $\mathcal C_z+I_N$ is invertible and all $F_{ab}$ vanish.  Hermitian
conjugation proves the minus relation.  For \eqref{eq:grass-cross}, write
$\overline{\mathcal C}_{\bar z}$ for the antiholomorphic Capelli matrix and
$\mathfrak s_j^-$ for the raw antiholomorphic column.  The two Capelli
matrices commute with the opposite raw contractions, and
$\{\mathfrak s_a^+,\mathfrak s_b^-\}=0$.  Hence
\[
 \{\mathfrak n_i^+,\mathfrak n_j^-\}
 =\sum_{a,b}(\mathcal C_z^{-1})_{ia}
   (\overline{\mathcal C}_{\bar z}^{-1})_{jb}
   \{\mathfrak s_a^+,\mathfrak s_b^-\}=0,
\]
which proves the cross relation explicitly.
\end{proof}

\subsection{Labelwise factorization and all filtrations}
For each label set
\begin{equation}\label{eq:Hipm}
 \mathbb H_i^\pm
 :=\mathfrak n_i^\pm\widehat\otimes\mathbb N_i^\pm,
 \qquad
 \mathbb H_q^\pm=\sum_{i=1}^N\mathbb H_i^\pm.
\end{equation}

\begin{proposition}[Commuting label factors]\label{prop:Hcommute}
For $\sigma,\tau\in\{+,-\}$,
\begin{equation}\label{eq:Hcommute}
 (\mathbb H_i^\sigma)^2=0,
 \qquad
 [\mathbb H_i^\sigma,\mathbb H_j^\tau]=0\quad(i\ne j).
\end{equation}
Consequently
\begin{equation}\label{eq:U-factor}
 \mathbb U_q^\pm=\prod_{i=1}^N(I+\mathbb H_i^\pm)
\end{equation}
and
\begin{equation}\label{eq:R-factor}
 \boxed{
 \mathbb R_q=\prod_{i=1}^N\mathbb R_i,
 \qquad
 \mathbb R_i:=(I+\mathbb H_i^+)(I-\mathbb H_i^-), }
\end{equation}
with mutually commuting $\mathbb R_i$ and
$\mathbb R_i^{-1}=(I+\mathbb H_i^-)(I-\mathbb H_i^+)$.
\end{proposition}

\begin{proof}
The square vanishes by Theorem~\ref{thm:grassmann} together with
$(\mathbb N_i^\pm)^2=0$.  For $i\ne j$ the bosonic contractions anticommute,
while the odd seed maps in different fermionic label factors anticommute by
the graded tensor convention; the two signs therefore cancel in the even
products.  No additional spectral condition is needed for this cancellation.  A
plus seed coefficient depends only on the barred Euler degree of its own
label and a minus coefficient only on the unbarred degree; moreover
\eqref{eq:n-covariance} gives
$[\mathcal E^z_{jj},\mathfrak n_i^+]=0$ for $i\ne j$, with the conjugate
identity in the barred polarization.  Thus the relevant shifted coefficients
commute with the opposite contraction in a different label.  Hence the even
label generators commute.  Their finite exponentials split as in
\eqref{eq:U-factor}, and rearranging the distinct-label factors gives
\eqref{eq:R-factor}.
\end{proof}

For an operator $X$ define the local even defect
\begin{equation}\label{eq:Ki}
 \mathscr K_i(X):=\operatorname{Ad}_{\mathbb R_i}(X)-X.
\end{equation}
No infinite BCH expansion is involved.  If
$a_i=\mathbb H_i^+$ and $b_i=\mathbb H_i^-$, then exactly
\begin{equation}\label{eq:local-defect}
 (I+\mathscr K_i)(X)
 =(I+a_i)(I-b_i)X(I+b_i)(I-a_i).
\end{equation}

\begin{theorem}[Complete relative connection]\label{thm:Omega-complete}
For arbitrary finite label sets and finite fermionic support in the stable range,
\begin{align}
 \boldsymbol\Omega_{j,q}^{-\mid+}
 &=\left[\prod_{i=1}^N(I+\mathscr K_i)-I\right](\mathbb A_j^-)
 \label{eq:Omega-product}\\
 &=\sum_{\varnothing\ne I\subseteq\{1,\ldots,N\}}
 \mathscr K_I(\mathbb A_j^-),
 \qquad
 \mathscr K_I:=\prod_{i\in I}\mathscr K_i.
 \label{eq:Omega-subsets}
\end{align}
The order in $\mathscr K_I$ is irrelevant.  Hence every term of filtration
$\ge2$ is a finite subset product of the same labelwise defects and carries
no new independent coefficient.
\end{theorem}

\begin{proof}
Proposition~\ref{prop:Hcommute} gives
$\operatorname{Ad}_{\mathbb R_q}
=\prod_i\operatorname{Ad}_{\mathbb R_i}
=\prod_i(I+\mathscr K_i)$.  Apply this identity to $\mathbb A_j^-$ and subtract
$\mathbb A_j^-$.  Expanding the finite commuting product gives
\eqref{eq:Omega-subsets}.
\end{proof}

For example, when $N=2$ the complete relative connection is already
\[
 \boldsymbol\Omega_{j,q}^{-\mid+}
 =\mathscr K_1(\mathbb A_j^-)+\mathscr K_2(\mathbb A_j^-)
  +\mathscr K_1\mathscr K_2(\mathbb A_j^-).
\]
The last summand is the entire filtration-two contribution; there is no
additional two-label coefficient or overlap tensor to determine.

The relative connection formula already closes the mixed algebra via
Theorem~\ref{thm:relative}.  For completeness, one can differentiate the
finite product without introducing an implicit remainder.  Put
\begin{equation}\label{eq:Theta}
 \boldsymbol\Theta_{k,i}^{+\mid-}
 :=\mathbb R_i^{-1}\mathbb A_k^+\mathbb R_i-\mathbb A_k^+,
\end{equation}
and write $\operatorname{ad}^{\rm s}_Y(X)=[Y,X]_{\rm s}$ for the superbracket.
Then
\begin{equation}\label{eq:deriv-K}
 [\operatorname{ad}^{\rm s}_{\mathbb A_k^+},\mathscr K_i]
 =\operatorname{Ad}_{\mathbb R_i}
  \operatorname{ad}^{\rm s}_{\boldsymbol\Theta_{k,i}^{+\mid-}}.
\end{equation}
For a fixed order of labels let
$\mathscr T_{<r}=\prod_{i<r}(I+\mathscr K_i)$,
$\mathscr T_{>r}=\prod_{i>r}(I+\mathscr K_i)$ and
$\mathscr T=\prod_i(I+\mathscr K_i)$.

\begin{theorem}[Closed all-filtration mixed table]\label{thm:mixed-closed}
In the ungauged plus seed frame,
\begin{equation}\label{eq:mixed-closed}
 \boxed{
 \mathbb L_{kj}^{\rm seed}
 +\{\mathbb A_k^+,\boldsymbol\Omega_{j,q}^{-\mid+}\}
 =\mathscr T(\mathbb L_{kj}^{\rm seed})
 +\sum_{r=1}^N
  \mathscr T_{<r}\operatorname{Ad}_{\mathbb R_r}
  \operatorname{ad}^{\rm s}_{\boldsymbol\Theta_{k,r}^{+\mid-}}
  \mathscr T_{>r}(\mathbb A_j^-). }
\end{equation}
Together with \eqref{eq:mixed-relative}, this determines every mixed
$\{\mathbb D_{k,q}^+,\mathbb D_{j,q}^-\}$ on arbitrary polynomial inputs.
No further higher-label overlap matrix or Euler-shell function is required.
\end{theorem}

\begin{remark}[Meaning of ``closed'']
Theorem~\ref{thm:mixed-closed} is not a claim that the mixed operators form a
finite-dimensional Lie superalgebra with scalar structure constants.  It is a
closure statement in the operator category used here: every mixed
anticommutator is an explicit finite composition of the already defined
Capelli contractions, four-slot and finite-support fermionic maps, Euler shifts and
classical mixed operators.  No new degreewise inversion or undetermined
coefficient is introduced by higher label filtration.
\end{remark}

\begin{proof}
Differentiate the finite commuting product
$\mathscr T=\prod_i(I+\mathscr K_i)$ by the odd derivation
$\operatorname{ad}^{\rm s}_{\mathbb A_k^+}$.  The zeroth-order term is
$\mathscr T(\mathbb L_{kj}^{\rm seed})$; when the derivation hits the $r$th
factor, use \eqref{eq:deriv-K}.  Summing the finitely many possibilities gives
\eqref{eq:mixed-closed}.
\end{proof}

Theorem~\ref{thm:mixed-closed} is an operator identity, not a formal
existence statement: every factor $\mathbb R_i$ is the finite four-slot
conjugation \eqref{eq:local-defect}, its derivative is controlled by
Proposition~\ref{prop:spin-cartan}, and the only one-label Euler shifts are
those displayed in Theorem~\ref{thm:local-Omega} and
Appendix~\ref{app:mixed-normal-order}.  Thus the formula can be evaluated on
an arbitrary polynomial without solving an additional overlap problem.

\subsection{Common complex structure and two full products}
Transport the undeformed localized Clifford--Weyl product by each total frame:
\begin{equation}\label{eq:starpm}
 F\star_q^\pm G
 :=\mathbb T_q^\pm\bigl((\mathbb T_q^\pm)^{-1}F\,
                         (\mathbb T_q^\pm)^{-1}G\bigr).
\end{equation}
Both are associative and unital and restrict to the common scalar product
\eqref{eq:scalar-product}.  The relative map
\begin{equation}\label{eq:Srel}
 \mathbb S_q:=(\mathbb T_q^-)^{-1}\mathbb T_q^+
 =\mathbb U_q^-(\mathbb U_q^+)^{-1}=\mathbb R_q^{-1}
\end{equation}
is multiplicative for the undeformed product if and only if
$\star_q^+=\star_q^-$.

Let $J$ be the classical Hermitian complex structure.  Every scalar Gram
factor is $J$-neutral; the Wick exponent removes one plus and one minus Witt
symbol while inserting a neutral Gram factor; and
$\mathfrak n_i^\pm\widehat\otimes\mathbb N_i^\pm$ replaces one bosonic factor
by a fermionic factor of the same Hermitian weight.  Therefore
\begin{equation}\label{eq:Jcommon}
 [\mathbb T_q^\pm,J]=0,
 \qquad
 J_q^+=J_q^-=J.
\end{equation}

For the multiplicativity test we use one explicit charge-zero boundary.
Fix one bosonic label and let $e_{\mathbf0}$ be the tensor product of the
empty-occupancy charge-zero vectors of the $n$ fermionic pairs.  Let
$e_{\vartheta,r}$ denote the vector with the $r$th pair in the
$\vartheta$-slot and all other pairs in the empty slot, and let
$\pi_{0;r}^+$ be the copy in the $r$th pair of the charge-zero coefficient
$\pi_0^+$ from the four-slot contracting homotopy of
Section~\ref{sec:fermionic}.  Finally, write $Z_{\rm b}$ for the single
linear bosonic Hermitian vector in this label.  The empty-occupancy branch
defect computed in Section~\ref{sec:fermionic}, together with
$\mathfrak n^+(Z_{\rm b})=1$, gives
\begin{equation}\label{eq:charge-zero-linear-defect}
 (\mathbb U_q^+)^{-1}(Z_{\rm b}e_{\mathbf0})
 =Z_{\rm b}e_{\mathbf0}
 -\frac{\delta_m(q)}{n}\sum_{r=1}^n
 e_{\vartheta,r}(\pi_{0;r}^+)^{-1},
 \qquad
 \delta_m(q):=\frac{2m-[2m]_q}{4}.
\end{equation}
This is the only local formula needed in the following obstruction.

\begin{theorem}[No common full transported product]\label{thm:no-common-product}
Assume $0<q<1$, $n\ge1$ and the stable hypotheses.  Then $\mathbb S_q$ is not
multiplicative for the undeformed localized Clifford--Weyl product.  Hence
\begin{equation}\label{eq:stars-distinct}
 \star_q^+\ne\star_q^-.
\end{equation}
Hermitian conjugation exchanges them:
\begin{equation}\label{eq:star-conj}
 (F\star_q^+G)^\ddagger=F^\ddagger\star_q^-G^\ddagger.
\end{equation}
\end{theorem}

\begin{proof}
It is enough to work in the one-label charge-zero block just described.
Only the plus Capelli contraction acts on the linear bosonic input.  After
that contraction the bosonic coefficient is scalar, so the minus bosonic
contraction annihilates it.  Hence the minus triangular factor cannot alter
the correction in \eqref{eq:charge-zero-linear-defect}, and
\begin{equation}\label{eq:S-linear}
 \mathbb S_q(Z_{\rm b}e_{\mathbf0})
 =Z_{\rm b}e_{\mathbf0}
 -\frac{\delta_m(q)}{n}\sum_{r=1}^n
 e_{\vartheta,r}(\pi_{0;r}^+)^{-1}.
\end{equation}
The inverses exist in the factorwise localization of
Theorem~\ref{thm:min-local}.  Every
purely fermionic Cartan element is fixed by $\mathbb S_q$.  In the undeformed
product $[Z_{\rm b},B_{{\rm f},r}]=0$, but the Weyl weight relation gives
\[
 [e_{\vartheta,r}R(B_{{\rm f},r}),B_{{\rm f},r}]
 =-2i\,e_{\vartheta,r}R(B_{{\rm f},r}).
\]
Hence
\begin{equation}\label{eq:mult-defect}
 [\mathbb S_q(Z_{\rm b}e_{\mathbf0}),
   \mathbb S_q(B_{{\rm f},r})]
 =\frac{2i\delta_m(q)}{n}
 e_{\vartheta,r}(\pi_{0;r}^+)^{-1}\ne0.
\end{equation}
For $0<q<1$, $\delta_m(q)\ne0$.  A multiplicative map would preserve the
zero commutator, contradiction.  Equation \eqref{eq:star-conj} follows by
conjugating \eqref{eq:starpm}; $\ddagger$ is order preserving.
\end{proof}

Thus the stable algebraic outcome is one common scalar transported product,
one common classical complex structure, and two conjugate flat full
Clifford--Weyl products whose discrepancy is measured by the same relative
transport that controls the mixed anticommutator.

\section{Locality obstructions and structural consequences}\label{sec:obstructions}

The transported calculus is polynomial preserving, but its exact radial
coefficients already rule out two natural undeformed notions of strong
coordinate locality.

\begin{definition}\label{def:local-classes}
An operator is \emph{ordinary differential local} if it is a fixed finite sum
\[
 T=\sum_{|\alpha|+|\beta|\le R}
 P_{\alpha,\beta}(z,\bar z)
 \partial_z^\alpha\partial_{\bar z}^\beta,
\]
with $R$ independent of polynomial degree.  It is \emph{finite nonzero-shift
local} if it belongs to the algebra generated by polynomial multipliers,
ordinary derivatives, Berezin operators, Clifford--Weyl coefficient maps that
are independent of the bosonic polynomial degree, and a fixed finite collection of nonzero coordinate dilations and Jackson
differences.  Euler inverses, continuous dilation averages, spectral
projectors and global Fischer decompositions belong to neither class.
\end{definition}

\begin{theorem}[No fixed finite-order differential realization]
\label{thm:no-differential}
For $0<q<1$, no ordinary differential local operator can satisfy the exact
rank-one scalar radial restriction
\begin{equation}\label{eq:scalar-chain-no-go}
 T(H^re)=c_r Z^\dagger H^{r-1}e,
 \qquad c_r=\frac12[2r]_q,
 \qquad r\ge1,
\end{equation}
for a fixed fermionic coefficient vector $e$.
\end{theorem}

\begin{proof}
Repeated differentiation of $H^r$ has the form
\[
 \partial_z^\alpha\partial_{\bar z}^\beta H^r
 =\sum_{k=0}^{|\alpha|+|\beta|}
 r^{\underline k}H^{r-k}Q_{\alpha,\beta,k}(z,\bar z),
\]
with $Q_{\alpha,\beta,k}$ independent of $r$.  Restrict to the first coordinate
axis and extract the coefficient of $z_1^{r-1}\bar z_1^r$ in a fixed output
component.  Any fixed finite-order differential expression produces a
polynomial function of $r$.  Exact radial restriction forces that polynomial
to equal
\[
 c_r=\frac{1-q^{2r}}{2(1-q)}
\]
for every positive integer $r$.  This sequence is bounded and strictly
increasing, whereas a polynomial bounded on the positive integers is
constant.  Contradiction.
\end{proof}

The scalar chain itself is compatible with one Jackson dilation.  The stronger
finite-shift obstruction comes from the occupied scalar coefficient
\begin{equation}\label{eq:beta}
 \beta_r:=\frac{c_{r+1}}{r+1}
 =\frac{1-q^{2r+2}}{2(1-q)(r+1)}.
\end{equation}

\begin{definition}[Exponential-polynomial sequence]
\label{def:exponential-polynomial}
A sequence $(a_r)_{r\ge0}$ is an \emph{exponential polynomial} if
\[
 a_r=\sum_{j=1}^J p_j(r)\lambda_j^r
\]
for finitely many polynomials $p_j\in\C[r]$ and nonzero
$\lambda_j\in\C$.
\end{definition}

\begin{lemma}[Finite shifts give exponential polynomials]
\label{lem:finite-shift-exponential}
Let $T$ be finite nonzero-shift local in the sense of
Definition~\ref{def:local-classes}.  Fix a coefficient-space input, an output
coefficient functional, and integers $a,b$.  After restriction to the first
coordinate axis, the coefficient of
$z_1^{r+a}\bar z_1^{r+b}$ in the selected component of
$T(z_1^r\bar z_1^r v)$ is an exponential polynomial in $r$ on every range
where the displayed powers are nonnegative.  In particular, its ordinary
generating function is rational.
\end{lemma}

\begin{proof}
It is enough to follow one word in the permitted generators.  Polynomial
multiplication shifts the exponents by fixed integers, whereas ordinary
partial derivatives contribute falling factorials in $r$.  A fixed dilation
contributes a factor $t^{r+s}$, and a Jackson difference contributes
$[r+s]_t=(1-t^{r+s})/(1-t)$ followed by a fixed exponent shift.  Hence each
generator preserves finite sums $\sum_jp_j(r)\lambda_j^r$, and so do finite
compositions and finite sums.  Finally,
\[
 \sum_{r\ge0}r^k\lambda^r x^r
 =\left(x\frac{d}{dx}\right)^k\frac{1}{1-\lambda x}
\]
is rational, proving the last assertion.
\end{proof}

\begin{lemma}[The occupied coefficient is not an exponential polynomial]
\label{lem:beta-not-exponential}
For $0<q<1$, the sequence $\beta_r$ in \eqref{eq:beta} is not an
exponential polynomial.
\end{lemma}

\begin{proof}
For $|x|<1$,
\begin{equation}\label{eq:beta-gf}
 \sum_{r\ge0}\beta_rx^r
 =\frac{1}{2(1-q)x}
  \log\!\left(\frac{1-q^2x}{1-x}\right).
\end{equation}
Indeed this follows from
$\sum_{r\ge0}x^r/(r+1)=-\log(1-x)/x$ and its $q^2x$ analogue.
The right-hand side has a logarithmic singularity at $x=1$ and is not
rational.  Lemma~\ref{lem:finite-shift-exponential} therefore rules out an
exponential-polynomial representation.
\end{proof}

\begin{theorem}[No finite nonzero-shift realization]\label{thm:no-shift}
For $0<q<1$, no finite nonzero-shift local operator on the undeformed
polynomial coordinate--Clifford--Weyl algebra can have the complete flat
rank-one one-pair radial restriction.
\end{theorem}

\begin{proof}
The exact occupied scalar chain contains the component with coefficient
$\beta_r$ in \eqref{eq:beta}.  Under the standard coordinate inclusion, the
bosonic radial variable is a fixed nonzero multiple of the Hermitian Gram
variable.  Restricting to the first coordinate axis therefore turns its
$r$th power into a fixed nonzero multiple of $z_1^r\bar z_1^r$.  Extract the
fixed occupied fermionic output component.  If the extending operator were
finite nonzero-shift local, Lemma~\ref{lem:finite-shift-exponential} would
make this coefficient an exponential polynomial in $r$.  Exact radial
restriction forces it to equal $\beta_r$, contradicting
Lemma~\ref{lem:beta-not-exponential}.
\end{proof}

\begin{proposition}[Continuous-dilation form of the forced denominators]
\label{prop:continuous-dilation}
The reciprocal spectral factors responsible for the rank-one Capelli/Euler normalization
and the occupied radial coupling have exact continuous-dilation formulas.
If $P$ is homogeneous of holomorphic degree $d$, then
\begin{equation}\label{eq:euler-resolvent-dilation}
 (m+E_z)^{-1}P
 =\int_0^1t^{m-1}P(tz,\bar z)\,dt,
\end{equation}
and
\begin{equation}\label{eq:beta-dilation}
 \beta_r=\frac{1}{2(1-q)}\int_{q^2}^{1}t^r\,dt.
\end{equation}
Thus the preceding obstruction concerns \emph{finite} shift support, not
dilation calculus itself.
\end{proposition}

\begin{proof}
The first identity is the integral of $t^{m+d-1}$; the second is
$\int_{q^2}^1t^r\,dt=(1-q^{2r+2})/(r+1)$.
\end{proof}

These results do not rule out other notions of local $q$-superspace.  They
show that an exact realization of the present flat radial theory cannot be
obtained inside the two undeformed coordinate classes considered above.  A
braided or quantum coordinate algebra would require a different construction.

\section{Conclusion}\label{sec:conclusion}

In the stable range we have constructed a $q$-Hermitian coordinate calculus on the localized transported module for arbitrary finite bosonic label sets and finite fermionic support.  The flat radial deformation extends to general Clifford-valued coordinates while preserving the Hermitian square-zero relations and the mixed polarization algebra.  The proof has three main steps.

First, stable Howe separation converts the scalar coordinate algebra into a
unique Gram--harmonic normal form.  The radial divided-power weights therefore
extend canonically to all scalar polynomials.  The label--Capelli inverse then
normalizes the coordinate Clifford contractions, and the Capelli--PBW
factorization identifies them with the abstract label contractions on all
standard radial exterior words.  A finite Wick--Chevalley conjugation lifts
the scalar gauge to the complete Clifford PBW module.

Second, the finite-support fermionic seed complexes are contractible over the
minimal factorwise Cartan localization.  Coupling their nilpotent homotopies
to the normalized Capelli contractions produces one finite triangular
transport in each polarization.  The resulting coordinate families are
square-zero, mutually anticommute within each polarization, reproduce the
flat radial PBW calculus exactly, carry scalar boundary trace $[2m-2n]_q$, and have the
classical limit.

Third, the global Capelli--Grassmann relations make the mixed problem finite.
Different labels contribute commuting even triangular factors, so the
relative transport factorizes and every higher-filtration mixed term is a
finite subset product of labelwise defects.  This closes the full mixed
$+/-$ algebra.  The same relative map also shows that the classical complex
structure survives unchanged, whereas one common full transported
Clifford--Weyl product does not.

The localization is forced by the algebraic boundary relations.  The one-pair
problem requires the Cartan denominator in the natural polynomial first-order
seed class, and the all-charge two-polarization problem requires the full
shift-stable orbit set.  Nevertheless, in the natural range $m\ge n$ the
$q$-dependent orbit factors are nonresonant on the standard polynomial
oscillator.  Thus the deformation does not create new oscillator poles; any
remaining all-charge resonance comes from the classical Weyl shift orbit.

We finally note two limitations of the present construction.  It is not a fixed local
differential calculus on ordinary superspace, and Theorems~\ref{thm:no-differential}
and \ref{thm:no-shift} explain why two natural undeformed locality classes
cannot realize the exact radial theory.  Also, the exceptional range $m<2N$
is not part of the present theorem: there the stable normal form acquires
additional syzygies and should be treated by resolution-theoretic methods.
Analytic function theory, covariance under a possible quantum unitary/Spin structure, and integral formulas likewise belong to later work.

\appendix
\section{Rank-one right-normal mixed correction}\label{app:mixed-normal-order}

This appendix supplies the normal-ordering step omitted from the factorized argument in Section~\ref{sec:mixed}.  It is included because the local
mixed correction is one of the places where an Euler-shell shift can easily
be missed.  Throughout this appendix there is one bosonic label and one
fermionic pair.  Put
\[
 d^+:=\partial_Z,
 \qquad d^-:=\partial_{Z^\dagger},
 \qquad e^\pm:=\{d^\pm,\mathfrak n^\pm\},
\]
and retain the one-pair seed maps
$\mathbb Q^\pm,\mathbb M^\pm,\mathbb N^\pm$ of
Section~\ref{sec:fermionic}.  All seed spectral coefficients are written on
the right.

For a right-normal seed coefficient $Y(E,\bar E)$ let
$\tau_z^{-1}Y:=Y(E-1,\bar E)$ and
$\tau_{\bar z}^{-1}Y:=Y(E,\bar E-1)$.  Define
\begin{equation}\label{eq:appendix-Lambdas}
 \Lambda^{+|-}:=\{\mathfrak n^+,d^-\},
 \qquad
 \Lambda^{-|+}:=\{d^+,\mathfrak n^-\}.
\end{equation}
For $N=1$, Proposition~\ref{prop:spin-cartan} gives
\begin{equation}\label{eq:appendix-dLambda}
 [d^+,\Lambda^{+|-}]
 =\varepsilon_z d^+\Delta,
 \qquad
 \varepsilon_z:=(m+E+1)^{-1}-(m+E)^{-1}.
\end{equation}
Similarly, with
$\mathfrak s^-=(m+\bar E)\mathfrak n^-$,
\begin{equation}\label{eq:appendix-Xi}
 \Xi^-:=[d^+,e^-]
 =\varepsilon_{\bar z}\Delta(\mathfrak s^- -d^-),
 \qquad
 \varepsilon_{\bar z}:=(m+\bar E+1)^{-1}-(m+\bar E)^{-1}.
\end{equation}

\begin{lemma}[Right-normal mixed anticommutator rule]
\label{lem:appendix-right-normal-rule}
Let $X\widehat\otimes Y(E,\bar E)$ be an odd right-normal operator term.
If $X$ is odd and $Y$ is even, then
\begin{align}
 \{\mathbb A^+,X\widehat\otimes Y\}
 &=\{d^+,X\}\widehat\otimes Y
   +Xd^+\widehat\otimes(\tau_z^{-1}Y-Y)
   +X\widehat\otimes[Y,\mathbb Q^+].
 \label{eq:appendix-rule-odd-even}
\end{align}
If $X$ is even and $Y$ is odd, then
\begin{align}
 \{\mathbb A^+,X\widehat\otimes Y\}
 &=[d^+,X]\widehat\otimes Y
   +Xd^+\widehat\otimes(Y-\tau_z^{-1}Y)
   +X\widehat\otimes\{\mathbb Q^+,Y\}.
 \label{eq:appendix-rule-even-odd}
\end{align}
\end{lemma}

\begin{proof}
The graded tensor signs give the first and third terms in each identity.  For
the middle term one uses
$Y(E,\bar E)d^+=d^+(\tau_z^{-1}Y)(E,\bar E)$ before restoring right-normal
order.  The fermionic seed preserves the bosonic Euler degrees, so no further
shift occurs.
\end{proof}

Abbreviate the five seed combinations occurring in
\eqref{eq:local-Omega} by
\begin{align}
 \mathbb S&:=\mathbb N^+-\tau_{\bar z}^{-1}\mathbb N^+,
 &\mathbb R&:=\{\mathbb N^+,\mathbb Q^-\},\notag\\
 \mathbb T&:=(\tau_z^{-1}\mathbb M^-)\mathbb N^+
             -(\tau_{\bar z}^{-1}\mathbb N^+)\mathbb M^-,
 &\mathbb U&:=(\tau_{\bar z}^{-2}\mathbb N^+)\mathbb N^-,\notag\\
 &&\mathbb V&:=(\tau_z^{-1}\mathbb N^-)\mathbb N^+.
 \label{eq:appendix-STURV}
\end{align}
Every object on the right is an explicit four-slot Cartan map because
$\mathbb Q^\pm,\mathbb M^\pm,\mathbb N^\pm$ were given in
Section~\ref{sec:fermionic}.

\begin{theorem}[Complete rank-one one-pair mixed correction]
\label{thm:appendix-complete-rank-one}
On every rank-one polynomial--localized one-pair block,
\begin{align}
 \{\mathbb A^+,\boldsymbol\Omega_q^{-|+}\}
={}&\Lambda^{-|+}\widehat\otimes\mathbb M^-
 +\mathfrak n^-d^+\widehat\otimes
   (\tau_z^{-1}\mathbb M^- -\mathbb M^-)
 +\mathfrak n^-\widehat\otimes[\mathbb M^-,\mathbb Q^+]
 \notag\\
&+\Xi^-\widehat\otimes\mathbb N^-
 +e^-d^+\widehat\otimes
   (\mathbb N^- -\tau_z^{-1}\mathbb N^-)
 +e^-\widehat\otimes\{\mathbb Q^+,\mathbb N^-\}
 \notag\\
&-[d^+,\Lambda^{+|-}]\widehat\otimes\mathbb N^+
 +\Lambda^{+|-}\widehat\otimes\mathbb M^+
 \notag\\
&+(e^+d^- -\mathfrak n^+\Delta)\widehat\otimes\mathbb S
 +(\mathfrak n^+d^-)\widehat\otimes
   \{\mathbb Q^+,\mathbb S\}
 \notag\\
&+e^+\widehat\otimes\mathbb R
 +\mathfrak n^+\widehat\otimes[\mathbb R,\mathbb Q^+]
 \notag\\
&+(e^+\mathfrak n^- -\mathfrak n^+\Lambda^{-|+})
   \widehat\otimes\mathbb T
 \notag\\
&+(\mathfrak n^+\mathfrak n^-)d^+\widehat\otimes
   (\mathbb T-\tau_z^{-1}\mathbb T)
 +(\mathfrak n^+\mathfrak n^-)\widehat\otimes
   \{\mathbb Q^+,\mathbb T\}
 \notag\\
&+(e^+e^- -\mathfrak n^+\Xi^-)
   \widehat\otimes\mathbb U
 +(\mathfrak n^+e^-)\widehat\otimes[\mathbb U,\mathbb Q^+]
 \notag\\
&+(\Xi^-\mathfrak n^+ +e^-e^+)
   \widehat\otimes\mathbb V
 \notag\\
&+(e^-\mathfrak n^+)d^+\widehat\otimes
   (\tau_z^{-1}\mathbb V-\mathbb V)
 +(e^-\mathfrak n^+)\widehat\otimes[\mathbb V,\mathbb Q^+].
 \label{eq:appendix-full-mixed}
\end{align}
Thus the full one-pair mixed anticommutator is obtained by adding
$\mathbb L^{\rm seed}$ and applying the outer transport in
\eqref{eq:mixed-relative}.  Every seed bracket in
\eqref{eq:appendix-full-mixed} is an explicit composition of the four-slot
maps of Section~\ref{sec:fermionic}; no further seed matrix or BCH remainder
is implicit.
\end{theorem}

\begin{proof}
Apply Lemma~\ref{lem:appendix-right-normal-rule} term by term to the eight
terms of \eqref{eq:local-Omega}.  The bosonic commutators reduce to
\begin{align*}
 [d^+,\mathfrak n^+d^-]&=e^+d^- -\mathfrak n^+\Delta,\\
 [d^+,\mathfrak n^+\mathfrak n^-]
 &=e^+\mathfrak n^- -\mathfrak n^+\Lambda^{-|+},\\
 \{d^+,\mathfrak n^+e^-\}
 &=e^+e^- -\mathfrak n^+\Xi^-,\\
 \{d^+,e^-\mathfrak n^+\}
 &=\Xi^-\mathfrak n^+ +e^-e^+.
\end{align*}
The relation
$\{\mathbb Q^+,\mathbb N^+\}=-\mathbb M^+$ produces the positive
$\Lambda^{+|-}\widehat\otimes\mathbb M^+$ term.  The only additional
holomorphic shell shifts are those displayed explicitly in
\eqref{eq:appendix-full-mixed}; all other seed coefficients are unchanged by
$d^+$.  Equations \eqref{eq:appendix-dLambda}--\eqref{eq:appendix-Xi} control
the two inverse-Capelli commutators.  Collecting the resulting terms gives
\eqref{eq:appendix-full-mixed}.
\end{proof}

The formula also displays the required range of Euler shifts.  The one-pair seed
coefficients occurring in \eqref{eq:local-Omega} and
\eqref{eq:appendix-full-mixed} use only
\[
 E,E-1,E-2,
 \qquad
 \bar E,\bar E-1,\bar E-2.
\]
Hence the rank-one normal ordering terminates after two adjacent shells in
each polarization.  At $q=1$ every defect coefficient vanishes and the
correction reduces to zero, leaving the classical mixed operator.

\section*{Acknowledgements}
\enlargethispage{2\baselineskip}
This work was co-funded by the University of Ostrava, Grant No.~SGS05/P\v{R}F/2026. Yifan Zhang is also co-funded by the European Union under the REFRESH -- Research Excellence For REgion Sustainability and High-tech Industries project, No.~CZ.10.03.01/00/22003/0000048, via the Operational Programme Just Transition.

\section*{Statements and Declarations}
\noindent\textbf{Competing interests.}
The authors declare that they have no competing financial or non-financial interests that are directly or indirectly related to this work.

\noindent\textbf{Data availability.}
No datasets were generated or analysed during the current study.

\printbibliography
%%
%\bibliographystyle{amsplain}
%\bibliography{references}

\end{document}